\documentclass{siamart1116}
\renewcommand{\theequation}{\arabic{section}.\arabic{equation}}
\usepackage{graphicx}
\usepackage{amsmath}
\usepackage{amsfonts}
\usepackage{amssymb, amsfonts, amsmath, nicefrac,color}

\def\bbr{{\Bbb{R}}} 
\def\bbe{{\Bbb{E}}} 
 
\def\bbp{{\Bbb{P}}}

\def\bbd{{\Bbb{D}}}
\def\bbh{{\Bbb{H}}}

\newcommand{\N}{{\cal N}}

\newcommand{\T}{{\cal T}}
\newcommand{\R}{{\cal R}}
\newcommand{\D}{{\cal D}}

\newcommand{\U}{{\cal U}}
\newcommand{\V}{{\cal V}}

\newcommand{\F}{{\cal F}}

\newcommand{\C}{{\cal C}}

\newcommand{\B}{{\cal B}}

\newcommand{\X}{{\cal X}}

\newcommand{\Y}{{\cal Y}}

\newcommand{\s}{\mathcal{S}}
\def\w{\omega}
\newcommand{\rr}{\rightrightarrows}
\def\e{\varepsilon}

\newcommand{\cd}{{\mathfrak d}}
\newcommand{\crr}{{\mathfrak r}}
\newcommand{\cF}{{\mathfrak F}}

\newtheorem{assu}{Assumption}[section]
\newtheorem{example}{Example}[section]
\newtheorem{cor}{Corollary}[section]
\newtheorem{remark}{Remark}[section]
\newtheorem{alg}{Algorithm}[section]
\newcommand{\inmat}[1]{\mbox{\rm {#1}}}
\def\prob {{\sf Pr}}

\usepackage{lipsum}
\usepackage{amsfonts}
\usepackage{graphicx}
\usepackage{epstopdf}
\usepackage{algorithmic}
\usepackage{amsmath}
\ifpdf
  \DeclareGraphicsExtensions{.eps,.pdf,.png,.jpg}
\else
  \DeclareGraphicsExtensions{.eps}
\fi

\numberwithin{theorem}{section}

\newcommand{\TheTitle}{Convergence Analysis of Sample Average Approximation of 
Two-stage Stochastic Generalized Equations} 
\newcommand{\TheAuthors}{X. Chen, A. Shapiro and H. Sun}

\headers{Two-stage Stochastic Generalized Equations}{\TheAuthors}

\title{{\TheTitle}\thanks{Submitted to the editors DATE.
}}

\author{
  Xiaojun Chen\thanks{Department of Applied Mathematics, The Hong Kong Polytechnic University,
    (\email{xiaojun.chen@polyu.edu.hk}). Research of this author  was partly supported by Hong Kong Research Grant Council PolyU153016/16p.}
  \and
  Alexander Shapiro\thanks{School of Industrial and Systems Engineering, Georgia Institute of Technology, (\email{alex.shapiro@isye.gatech.edu}). Research of this author  was partly supported by  NSF grant 1633196 and  DARPA EQUiPS program, grant SNL 014150709.}
  \and
  Hailin Sun\thanks{School of Economics and Management, Nanjing University of Science and Technology;  Department of Applied Mathematics, The Hong Kong Polytechnic University, (\email{hlsun@njust.edu.cn}).  Research  of  this  author  was  partly  supported by  National Natural Science Foundation of China  11401308, 11571178. }
}

\usepackage{amsopn}

\ifpdf
\hypersetup{
  pdftitle={Convergence Analysis of Sample Average Approximation of
Two-state Stochastic Generalized Equations},
  pdfauthor={\TheAuthors}
}
\fi

\begin{document}

\maketitle

\begin{abstract}
  A solution of two-stage stochastic generalized equations is a pair: a first stage solution which is  independent of realization of the random data and a second stage solution which is a function of random variables.
 This paper studies   convergence of the sample average approximation
of   two-stage stochastic nonlinear  generalized equations. In particular an exponential rate of the convergence is shown   by using the perturbed partial linearization of functions. Moreover,  sufficient conditions for the existence, uniqueness, continuity and regularity of solutions of two-stage stochastic generalized equations are presented  under an  assumption of monotonicity of the involved  functions. These theoretical results are given without assuming relatively  complete recourse,
 and are  illustrated by two-stage stochastic non-cooperative games of two players.
\end{abstract}



\section{Introduction}
Consider the following two-stage Stochastic Generalized\linebreak
 Equations (SGE)
\begin{eqnarray}
\label{geq-1}
&&0\in \bbe[\Phi(x, y(\xi), \xi)]+ \Gamma_1(x),\;x\in X, \\
\label{geq-2}
&&0\in\Psi(x, y(\xi), \xi)+\Gamma_2(y(\xi), \xi), \;\;  \mbox{ for a.e.}\; \xi\in\Xi.
\end{eqnarray}
Here $X\subseteq \bbr^n$ is a nonempty closed convex set,  $\xi:\Omega\to\bbr^d$ is a random vector  defined on a probability space $(\Omega,\F,\bbp)$,   whose probability distribution $P=\bbp\circ \xi^{-1}$ is supported on  set $\Xi:=\xi(\Omega)\subseteq \bbr^d$, $\Phi:\bbr^n\times\bbr^m\times \bbr^d\to \bbr^n$ and
$\Psi:\bbr^n\times\bbr^m\times \bbr^d\to \bbr^m$,
and $\Gamma_1:\bbr^n\rr\bbr^n$, $\Gamma_2:\bbr^m\times \Xi\rr\bbr^m$ are multifunctions (point-to-set mappings). We assume throughout the paper that $\Phi(\cdot,\cdot,\xi)$ and $\Psi(\cdot,\cdot,\xi)$ are  {\em  Lipschitz continuous} with Lipschitz modules $\kappa_{\Phi}(\xi)$ and $\kappa_{\Psi}(\xi)$,   and $y(\cdot) \in {\cal Y}$ with ${\cal Y}$ being  the space of measurable functions from $\Xi$ to $\bbr^m$ such that the expected value in \eqref{geq-1} is well defined.

Solutions of \eqref{geq-1}--\eqref{geq-2} are searched over  $x\in X$ and $y(\cdot)\in {\cal Y}$ to
satisfy the corresponding inclusions, where the second stage inclusion \eqref{geq-2} should hold for almost every (a.e.) realization of $\xi$.
The first stage decision $x$ is made before observing realization of the random data vector $\xi$ and the second stage decision $y(\xi)$ is a function of $\xi$.

When the multifunctions $\Gamma_1$ and $\Gamma_2$ have the following form
$$\Gamma_1(x):=\N_C(x)\;\;\;{\rm and}\;\;\;\Gamma_2(y, \xi):=\N_{K(\xi)}(y),$$
where  $\N_C(x)$ is  the  normal cone to a nonempty closed convex set $C\subseteq \bbr^n$ at $x$, and similarly  for $\N_{K(\xi)}(y)$,
  the SGE \eqref{geq-1}--\eqref{geq-2} reduce to the two-stage Stochastic Variational Inequalities (SVI) as  in \cite{CPW,RW2017}. The two-stage SVI represent  first order optimality conditions for the two-stage stochastic optimization problem \cite{Birge,SDR} and model several equilibrium problems in stochastic environment \cite{CPW,CWZ}.
Moreover, if the sets $C$ and $K(\xi)$, $\xi\in \Xi$, are closed convex {\em cones}, then
\[
\N_C(x)=\{x^*\in  C^*:  x^\top x^* =0\},\;x\in C,
\]
where $C^*=\{x^*:  x^\top x^* \le 0,\;\forall x\in C\}$ is the (negative)   dual of cone $C$. 
In that case   the SGE \eqref{geq-1}--\eqref{geq-2} reduce to the following  two-stage stochastic cone VI
\begin{eqnarray*}
&&C\ni x\perp \bbe[\Phi(x, y(\xi), \xi)]\in - C^*,\;x\in X, \\
&&K(\xi)\ni y(\xi)\perp \Psi(x, y(\xi), \xi)\in - K^*(\xi), \;\;  \mbox{ for a.e.}\; \xi\in\Xi.
\end{eqnarray*}
In particular when $C:=\bbr^n_+$ with $C^*=-\bbr^n_+$, and $K(\xi):= \bbr^m_+$ with $K^*(\xi)=-\bbr^m_+$ for all $\xi\in\Xi$, the SGE \eqref{geq-1}--\eqref{geq-2} reduce  to the two-stage Stochastic Nonlinear Complementarity Problem (SNCP):
\begin{eqnarray*}
&&0\le  x\perp  \bbe[\Phi(x, y(\xi), \xi)] \ge 0, \\
&&0\le  y(\xi)\perp  \Psi(x, y(\xi), \xi) \ge 0, \,\,  \mbox{ for a.e.}\; \xi\in\Xi,
\end{eqnarray*}
which is a generalization of the two-stage Stochastic Linear Complementarity Problem (SLCP):
\begin{eqnarray}
\label{geq-vi-31}
&&0\le x\perp Ax + \bbe[B(\xi)y(\xi)] +q_1 \ge 0, \\
\label{geq-vi-41}
&&0\le y(\xi)\perp L(\xi)x +M(\xi)y(\xi) +q_2(\xi) \ge 0, \,\,  \mbox{ for a.e.}\; \xi\in\Xi,
\end{eqnarray}
where $A\in \bbr^{n\times n}$, $B: \Xi \to \bbr^{n\times m},$ $L: \Xi \to \bbr^{m\times n},$
$M: \Xi \to \bbr^{m\times m},$ $q_1\in \bbr^n, q_2: \Xi \to \bbr^{m}.$ The two-stage SLCP arises from
the KKT condition for the two-stage stochastic linear programmming \cite{CPW}. Existence of solutions of (\ref{geq-vi-31})-(\ref{geq-vi-41}) has been studied in \cite{ChSuXu17}. Moreover, the progressive hedging method has been applied to solve (\ref{geq-vi-31})-(\ref{geq-vi-41}), with a finite number of realizations of $\xi$, in \cite{RS}.

Most existing studies for two-stage stochastic problems assume {\em relatively complete recourse}, that is, for every $x\in X$ and a.e. $\xi\in \Xi$ the
second stage problem has at least one solution. However, for the SGE \eqref{geq-1}--\eqref{geq-2},
it could happen that for a certain  first stage decision $x\in X$, the second stage generalized equation
\begin{equation}\label{eq:geq-2xi}
0 \in \Psi(x, y, \xi)+\Gamma_2(y, \xi)
\end{equation}
 does not have a solution for some $\xi\in \Xi$. For such $x$ and  $\xi$ the second stage decision $y(\xi)$ is not defined.  If this happens for $\xi$ with   positive probability, then the expected value of the first stage problem is not defined and such $x$ should be avoided.

In this paper, without assuming {\em relatively complete recourse}, we study convergence of the Sample Average Approximation (SAA)
\begin{eqnarray}
\label{geq-7}
&&0\in N^{-1}\sum_{j=1}^N \Phi(x, y_j, \xi^j) + \Gamma_1(x),\;x\in X, \\
\label{geq-8}
&&0\in\Psi(x, y_j, \xi^j)+\Gamma_2(y_j, \xi^j), \;\;  j=1,...,N,
\end{eqnarray}
of the two-stage SGE \eqref{geq-1}--\eqref{geq-2} with $y_j$ being a copy of the second stage vector for $\xi=\xi^j$, $j=1,...,N$, where $\xi^1,...,\xi^N$ is an independent identically distributed (iid)  sample of random vector $\xi$.
The  paper is organized as follows. In section \ref{sec-theor} we investigate almost sure and  exponential rate of   convergence of solutions of the  sample average approximations of the two-stage SGE. In section \ref{sec-svi}  convergence analysis of  the mixed two-stage SVI-NCP is presented. In particular we give   sufficient conditions for the existence, uniqueness, continuity and regularity of solutions   by using the perturbed linearization of functions $\Phi$ and $\Psi$.  Theoretical results,   given in sections \ref{sec-theor} and \ref{sec-svi},  are
illustrated by numerical examples, using the  Progressive Hedging Method (PHM), in section \ref{sec-examp}. It is worth noting that  PHM is well-defined for two-stage monotone SVI without relatively complete recourse.
Finally   section \ref{sec-concl} is devoted to   conclusion remarks.


We use the following notation and terminology  throughout the paper.
 Unless stated otherwise $\|x\|$ denotes the Euclidean norm of vector $x\in \bbr^n$. By $\B:=\{x:\|x\|\le 1\}$ we denote unit ball in a considered  vector space.
For two sets $A,B\subset \bbr^m$  we  denote by $d(x,B):=\inf_{y\in B}\|x-y\|$ distance from a point $x\in \bbr^m$ to the set $B$, $d(x,B)=+\infty$ if $B$ is empty,    by $\bbd(A,B):=\sup_{x\in A}d(x,B)$ the deviation of set $A$ from the set $B$, and
$\bbh(A,B):=\max\{\bbd(A,B),\bbd(B,A)\}$. The indicator function of a set $A$ is defined as $I_A(x)=0$ for $x\in A$ and $I_A(x)=+\infty$ for $x\not\in A$. By ${\rm bd}(A)$, {\rm int}(A) and ${\rm cl}(A)$
we denote the boundary, interior  and topological closure of a set  $A\subset \bbr^m$. By  ${\rm reint}(A)$ we  denote  the  relative interior of a convex  set $A\subset \bbr^m$.
A  multifunction (point-to-set mappings)
$\Gamma:\bbr^n\rr\bbr^m$ assigns to a point $x\in \bbr^n$ a set $\Gamma(x)\subset \bbr^m$.
A multifunction $\Gamma:\bbr^n\rr\bbr^m$ is said to be {\em closed} if $x_k\to x$, $x_k^*\in \Gamma (x_k)$ and $x_k^*\to x^*$, then $x^*\in \Gamma (x)$. It is said that a multifunction  $\Gamma:\bbr^n\rr\bbr^n$ is {\em  monotone}, if
$
(x-x')^\top (y-y') \ge 0,
$
 for all $x,x'\in \bbr^n$, and  $y\in \Gamma(x)$, $y'\in \Gamma(x')$. Note that for a nonempty closed  convex set $C$, the normal cone multifunction $\Gamma(x):= \N_C(x)$ is closed and monotone.
Note also that the normal cone   $\N_C(x)$, at $x\in C$, is the (negative) dual of the tangent cone $\T_C(x)$. We use the same notation for $\xi$ considered as a random vector and as a variable $\xi \in \bbr^d$. Which of these two meanings is  used will be clear from the context.

\section{Sample average approximation  of the two-stage SGE}
\label{sec-theor}
\setcounter{equation}{0}

In this section we discuss statistical properties of the first stage solution  $\hat{x}_N$ of the SAA problem \eqref{geq-7}--\eqref{geq-8}. In particular  we investigate conditions ensuring convergence of $\hat{x}_N$,   with probability one (w.p.1) and exponential, to its counterpart of the true problem \eqref{geq-1}--\eqref{geq-2}.

Denote by $\X$ the set of $x\in X$ such that the second stage generalized  equation \eqref{eq:geq-2xi} has a solution for a.e. $\xi\in \Xi$.
The condition of relatively complete recourse means that $\X=X$.
Note   that $\X$ is a subset of $X$, and  if $(\bar{x},\bar{y}(\cdot))$ is a solution of \eqref{geq-1}--\eqref{geq-2}, then $\bar{x}\in \X$.
It is possible to formulate the two-stage SGE \eqref{geq-1}--\eqref{geq-2} in the following equivalent way.  Let $\hat{y}(x,\xi)$ be a solution function of the second stage problem \eqref{eq:geq-2xi} for $x\in \X$ and $\xi\in \Xi$, i.e.,
\[
0\in\Psi(x, \hat{y}(x,\xi), \xi)+\Gamma_2(\hat{y}(x,\xi),\xi), \;\;x\in \X,\;  \mbox{a.e.}\; \xi\in\Xi.
\]
 Then the first stage problem becomes
\begin{equation}\label{geq-3}
0\in\bbe[\Phi(x, \hat{y}(x,\xi), \xi)]+ \Gamma_1(x),\;x\in \X.
\end{equation}
If $\bar{x}$ is a solution of \eqref{geq-3}, then $(\bar{x},\hat{y}(\bar{x},\cdot))$ is a solution of \eqref{geq-1}--\eqref{geq-2}.
 Conversely if $(\bar{x},\bar{y}(\cdot))$ is a solution of \eqref{geq-1}--\eqref{geq-2}, then
 $\bar{x}$ is a solution of \eqref{geq-3}. Note that problem \eqref{geq-3} is a generalized equation which has been studied in the past decades, e.g. \cite{LRX14,Ro80,RWVA,sha2003}.

 It could happen that the second stage problem \eqref{eq:geq-2xi} has more than one   solution for some $x\in \X$.
In that case choice of $\hat{y}(x,\xi)$ is somewhat arbitrary. This motivates the following condition.
\begin{assu}
\label{ass-1}
For every   $(x,\xi)\in \X\times \Xi$, problem \eqref{eq:geq-2xi}  has a {\rm unique} solution.
\end{assu}
Under Assumption \ref{ass-1} the value $\hat{y}(x,\xi)$ is uniquely defined for all $x\in \X$ and   $\xi\in \Xi$, and the first stage problem \eqref{geq-3} can be written as the following generalized equation
\begin{equation}\label{geq-3a}
 0\in \phi(x)+ \Gamma_1(x),\;x\in \X,
\end{equation}
where
\begin{equation}\label{geq-3b}
\hat{\Phi}(x,\xi):= \Phi(x,\hat{y}(x,\xi),\xi) \,\;{\rm and}\;\, \phi(x):=\bbe[\hat{\Phi}(x,\xi)].
\end{equation}

If the SGE have relatively complete recourse, then under
  Assumption \ref{ass-1}
the   SAA problem \eqref{geq-7}--\eqref{geq-8} can be written as
\begin{equation}\label{saa-1}
 0\in\hat{\phi}_N(x)+ \Gamma_1(x),\;x\in X,
\end{equation}
where $\hat{\phi}_N(x):=N^{-1}\sum_{j=1}^N \hat{\Phi}(x,\xi^j)$ with   $\hat{\Phi}(x,\xi)$ defined in   \eqref{geq-3b}.
Problem \eqref{saa-1} can be viewed as the SAA of the first stage problem \eqref{geq-3a}.  If $(\hat{x}_N,\hat{y}_{jN})$ is a solution of the SAA problem \eqref{geq-7}--\eqref{geq-8},
then $\hat{x}_N$ is a solution of \eqref{saa-1} and
$\hat{y}_{jN}=\hat{y}(\hat{x}_N,\xi^j)$, $j=1,...,N$. Note that
the   SAA problem \eqref{geq-7}--\eqref{geq-8} can be considered
without assuming  the relatively complete recours, although in that case  it could happen that $\hat{\phi}_N(x)$ is not defined for some $x\in X\setminus \X$ and   solution $\hat{x}_N$ of \eqref{geq-7} is not implementable at the second stage for some realizations of the random vector $\xi$. Our aim is the convergence analysis of the SAA problem   \eqref{geq-7}--\eqref{geq-8}   when sample size $N$ increases.

Denote by $\s^*$ the set of solutions of the first stage problem \eqref{geq-3a} and by $\hat{\s}_N$   the set of solutions of the SAA  problem \eqref{geq-7} (in case of relatively complete recourse, $\hat{\s}_N$ is the set of solutions of problem   \eqref{saa-1} as well).
\begin{itemize}
  \item [$\bullet$]
By $\bar{\X}(\xi)$ we  denote the set of  $x\in X$ such  that problem \eqref{eq:geq-2xi} has a solution, and by  $\bar{\X}_N:=\cap_{j=1}^N\bar{\X}(\xi^j)$ the set   of  $x$ such that  problems \eqref{geq-8} have a solution.
\end{itemize}
Note that the set  $\X$ is equal to the intersection of $\bar{\X}(\xi)$, a.e. $\xi\in \Xi$; i.e., $\X=\cap_{\xi\in \Xi\setminus \Upsilon}\bar{\X}(\xi)$ for some set $\Upsilon\subset \Xi$ such that $P(\Upsilon)=0$.
Note also that if the  two-stage SGE have relatively complete recourse, then $\bar{\X}(\xi)=X$ for a.e. $\xi\in \Xi$.

\subsection{Almost sure convergence}

Consider the solution $\hat{y}(x,\xi)$ of the second stage problem \eqref{eq:geq-2xi}. To ensure continuity   of $\hat{y}(x,\xi)$ in $x\in \X$ for $\xi\in \Xi$, in addition to Assumption \ref{ass-1}, we need  the following
boundedness condition.
\begin{assu}
\label{ass-2}
For every $\xi\in\Xi$ and  $x\in \bar{\X}(\xi)$ there is a neighborhood $\V$ of $x$ and a measurable function $v(\xi)$ such that $\|\hat{y}(x',\xi)\|\le v(\xi)$ for all $x'\in \V\cap \bar{\X}(\xi)$.
\end{assu}

\begin{lemma}
\label{lm-cont}
Suppose that  Assumptions {\rm \ref{ass-1} and \ref{ass-2}} hold, and   for every $\xi\in \Xi$ the multifunction $\Gamma_2(\cdot,\xi)$ is closed.   Then for every $\xi\in \Xi$ the solution $\hat{y}(x,\xi)$ is a continuous function of $x\in \X$.
 \end{lemma}

\begin{proof}
The proof is quite standard. We argue by a contradiction. Suppose that for some $\bar{x}\in \X$ and $\xi\in \Xi$ the solution $\hat{y}(\cdot,\xi)$ is not continuous at $\bar{x}$. That is, there is a sequence $x_k\in \X$ converging to $\bar{x}\in \X$ such that $y_k:=\hat{y}(x_k,\xi)$ does not converge to $\bar{y}:=\hat{y}(\bar{x},\xi)$. Then by the boundedness assumption, by passing to a subsequence if necessary we can assume that $y_k$ converges to a point $y^*$ different from  $\bar{y}$. Consequently $0\in \Psi(x_k,y_k,\xi)+ \Gamma_2(y_k,\xi)$ and $\Psi(x_k,y_k,\xi)$ converges
 to $\Psi(\bar{x},y^*,\xi)$. Since  $\Gamma_2(\cdot,\xi)$ is closed, it follows that    $0\in\Psi(\bar{x},y^*,\xi)+ \Gamma_2(y^*,\xi)$.  Hence by the uniqueness assumption, $y^*=\bar{y}$  which gives the required contradiction.
\end{proof}

Suppose for the moment that in addition to the assumptions of Lemma \ref{lm-cont},  the SGE have relatively complete recourse.
We can apply then  general results  to verify consistency of the SAA estimates.
Consider function  $\hat{\Phi}(x,\xi)$ defined in  \eqref{geq-3b}.
By continuity of  $\Phi (\cdot,\cdot,\xi)$ and $\hat{y}(\cdot,\xi)$, we have that   $\hat{\Phi}(\cdot,\xi)$ is continuous on $X$. Assuming further that there is a compact set $X'\subseteq X$ such that $\mathcal{S}^*\subseteq X'$ and  $\|\hat{\Phi}(x,\xi)\|_{x\in X'}$  is dominated by an integrable function, we have that
the function $\phi(x)=\bbe[\hat{\Phi}(x,\xi)]$ is continuous on $X'$  and $\hat{\phi}_N(x)$ converges w.p.1 to $\phi(x)$
uniformly on   $X'$. Note that the boundedness condition of Lemma \ref{lm-cont} and continuity of $\Phi(\cdot,\cdot,\xi)$ imply that $\hat{\Phi}(\cdot,\xi)$ is bounded on $X'$.
Then   consistency of SAA solutions follows by \cite[Theorem 5.12]{SDR}.
We give below a more general result without the assumption of relatively complete recourse.

\begin{lemma}\label{l:closed}
Suppose that Assumptions {\rm \ref{ass-1} and \ref{ass-2}} hold.
Then for every  $\xi\in \Xi$ the set $\bar{\X}(\xi)$ is closed.
\end{lemma}
\begin{proof}
For a given $\xi\in \Xi$ let $x_k\in  \bar{\X}(\xi)$ be a sequence converging to a point $\bar{x}$. We need to show that $\bar{x}\in \bar{\X}(\xi)$.
Let  $y_k$ be the solution of \eqref{eq:geq-2xi} for $x=x_k$ and $\xi$. Then by Assumption \ref{ass-2}, there is    a neighborhood $\V$ of $\bar{x}$ and a measurable function $v(\xi)$  such that    $\|y_k\|\leq v(\xi)$  when $x_k\in \V$. Hence  by passing to a subsequence we can assume that $y_k$ converges to a point $\bar{y}\in \bbr^m$. Since $\Psi(\cdot,\cdot,\xi)$ is continuous and $\Gamma_2(\cdot,\xi)$ is  closed it follows that $\bar{y}$ is a solution of \eqref{eq:geq-2xi} for $x=\bar{x}$, and hence $\bar{x}\in \bar{\X}(\xi)$.
\end{proof}

By saying that a property holds w.p.1 for $N$ large enough we mean that there is a   set $\Sigma\subset \Omega$ of $\bbp$-measure zero such that for every $\w\in \Omega\setminus \Sigma$ there exists a positive integer  $N^*=N^*(\w)$ such that the property holds for all $N\ge N^*(\w)$ and $\w\in \Omega\setminus \Sigma$.

\begin{assu}
\label{ass-3}
For  any  $\delta\in(0,1)$, there exists a compact set $\bar{\Xi}_\delta\subset\Xi$ such that $\bbp(\bar{\Xi}_\delta)\geq 1-\delta$ and  the multifunction $\Delta_\delta: X\rr \bar{\Xi}_\delta$,
\begin{equation}\label{eq:Delta}
\Delta_\delta(x) : = \{\xi\in\bar{\Xi}_\delta: x\in\bar{\X}(\xi)\},
\end{equation}
 is upper semicontinuous. 
\end{assu}

The following lemma shows this assumption holds under mild conditions.

\begin{lemma}
Suppose $\Psi(\cdot,\cdot,\cdot)$ is continuous, $\Gamma_2(\cdot,\cdot)$ is closed and Assumption {\rm \ref{ass-2}} holds.
Then  $\Delta_\delta(\cdot)$ is upper semicontinuous. 
\end{lemma}

\begin{proof}
Consider the second stage generalized equation \eqref{geq-2} and any sequence $\{(x_k, y_k, \xi_k)\}$ such that $x_k\in X$, $\xi_k\in \Delta_\delta(x_k)$ with a corresponding second stage  solution $y_k$ and $(x_k,\xi_k)\to(x^*,\xi^*)\in X\times \Xi$. Since $\Psi$ is continuous w.r.t. $(x,y,\xi)$ and $\Gamma_2(\cdot,\cdot)$ is closed, we have that under Assumption~\ref{ass-2}, $\{y_k\}$ has accumulation points and any accumulation point $y^*$  satisfies
$$
0\in \Psi(x^*, y^*, \xi^*) + \Gamma_2(y^*, \xi^*),
$$
which implies $\xi^*\in\Delta_\delta(x^*)$. This shows that the multifunction $\Delta_\delta(\cdot)$ is closed. Since $\bar{\Xi}_\delta$ is compact, it follows that $\Delta_\delta(\cdot)$ is upper semicontinuous. 
\end{proof}

Note that in the case when  $\Xi$ is compact, we  can set $\delta=0$ and replace $\bar{\Xi}_\delta$ by $\Xi$.

\begin{theorem}
\label{th-consist}
Suppose that: {\rm (i)} Assumptions {\rm \ref{ass-1}-\ref{ass-3}} hold,
 {\rm (ii)} the multifunctions $\Gamma_1(\cdot)$ and $\Gamma_2(\cdot,\xi)$, $\xi\in \Xi$, are  closed, {\rm (iii)} there is a compact subset $X'$ of $X$ such that $\s^*\subset X'$ and w.p.1 for all $N$ large enough the set $\hat{\s}_N$ is nonempty and is contained in $X'$, {\rm (iv)} $\|\hat{\Phi}(x,\xi)\|_{x\in \X}$  is dominated by an integrable function, {\rm (v)} the set $\X$ is nonempty. Let
$
 \cd_N := \bbd\big( \bar{\X}_N\cap X', \X\cap X' \big).
$
Then the following statements hold.
\begin{itemize}
\item[{\rm (a)}]
 $\cd_N\to0$ and $\bbd(\hat{\s}_N,\s^*)\to 0$ w.p.1 as $N\to \infty$.

 \item[{\rm (b)}]
In addition assume that: {\rm (vi)}   for any   $\delta>0$, $\tau>0$ and a.e. $\w\in \Omega$, there exist $\gamma>0$   and $N_0=N_0(\w)$ such that for any $x\in \X\cap X' + \gamma\,\B$  and $N\geq N_0$, there exists $z_x\in \X\cap X'$ such that\footnote{Recall that $\hat{\phi}_N(x)=\hat{\phi}_N(x,\w)$ are random functions defined on the probability space $(\Omega,\F,\bbp)$.}
\begin{equation}\label{eq:convi1}
 \|z_x - x\|\leq \tau,\;\; \Gamma(x) \subseteq\Gamma_1(z_x) + \delta\B,\;\; {\rm and}\;\; \|\hat{\phi}_N(z_x) - \hat{\phi}_N(x)\|\leq \delta.
\end{equation}
Then w.p.1 for $N$   large enough it follows that
\begin{equation}\label{eq:dist}
\bbd(\hat{\mathcal{S}}_N, \mathcal{S}^*) \leq \tau +  \R^{-1}\left(\,\sup_{x\in \X\cap X'} \|\phi(x) - \hat{\phi}_N(x)\|\right),
\end{equation}
where  for $\e\ge 0$ and  $t\ge 0$,
\begin{equation*}
\R(\e):= \inf_{x\in \X\cap X',\, d(x, \s^*)\geq\e} d\big(0, \phi(x) + \Gamma_1(x)\big),
\end{equation*}
\begin{equation*}
 \R^{-1}(t): = \inf\{ \e\in\bbr_+: \R(\e) \ge  t \}.
\end{equation*}
\end{itemize}
\end{theorem}

\begin{proof}
Part (a).
Let $\xi^j=\xi^j(\w)$, $j=1,...,$ be the iid sample, defined on the
probability space $(\Omega,\F,\bbp)$,  and $\bar{\X}_N=\bar{\X}_N(\w)$
be the corresponding feasibility set of the SAA problem. Consider a
point $\bar{x}\in X'\setminus \X$ and its   neighborhood
$\V_{\bar{x}}=\bar{x}+\gamma \B$ for some $\gamma>0$. We have that
probability  $p:=\bbp\{\xi\in \Xi:\bar{x}\not \in \bar{\X}(\xi)\}$ is
positive.
Moreover it follows by
Assumption \ref{ass-3}  that we can choose $\gamma >0$  such that
probability
$\bbp\left\{\V_{\bar{x}}\cap \bar{\X}(\xi) = \emptyset\right\}$ is
positive. Indeed, for $\delta:=p/4$ consider the multifunction
$\Delta_\delta$ defined in \eqref{eq:Delta}.   By upper semicontinuity of $\Delta_\delta$ we have
that for any $\e>0$ there is $\gamma>0$ such that for all $x\in
\V_{\bar{x}}$ it follows that $\Delta_\delta(x)\subset
\Delta_\delta(\bar{x})+\e\B$. That is
\[
\cup_{x\in  \V_{\bar{x}}}\{\xi\in \bar{\Xi}_\delta:x\in
\bar{\X}(\xi)\}\subset \{\xi\in \bar{\Xi}_\delta:\bar{x}\in
\bar{\X}(\xi)\}+\e\B\subset  \{\xi\in \Xi:\bar{x}\in \bar{\X}(\xi)\}+\e\B.
\]
It follows that we can choose $\e>0$  small enough such that
\[
\bbp\big(\cup_{x\in  \V_{\bar{x}}}\{\xi\in \bar{\Xi}_\delta:x\in
\bar{\X}(\xi)\}\big)\le 1-p/2.
\]
Since  $\delta=p/4$ we obtain
\[
\bbp\big(\cup_{x\in  \V_{\bar{x}}}\{\xi\in \Xi:x\in
\bar{\X}(\xi)\}\big)\le 1-p/4.
\]
Noting that the event $\left\{\V_{\bar{x}}\cap \bar{\X}(\xi) =
\emptyset\right\}$ is complement of the event $\big\{\cup_{x\in
\V_{\bar{x}}}\{\xi\in\Xi:x\in \bar{\X}(\xi)\}\big\}$,  we obtain that
$\bbp\left\{\V_{\bar{x}}\cap \bar{\X}(\xi) = \emptyset\right\}\ge p/4$.

Consider the event
   $
   E_N:=\left\{\V_{\bar{x}}\cap \bar{\X}_N\ne \emptyset\right\}.
   $
   The complement of this event is $E^c_N=\left\{\V_{\bar{x}}\cap
\bar{\X}_N= \emptyset\right\}$.
   Since the sample  $\xi^j$, $j=1,...,$ is   iid, we have
\[
\begin{array}{lll}
\bbp\left\{\V_{\bar{x}}\cap \bar{\X}_N\ne
\emptyset\right\}&\le&\prod_{j=1}^N \bbp\left\{\V_{\bar{x}}\cap
\bar{\X}(\xi^j)\ne \emptyset\right\}\\
&=&\prod_{j=1}^N \left(1-
\bbp\left\{\V_{\bar{x}}\cap \bar{\X}(\xi^j) =
\emptyset\right\}\right)\leq(1-p/4)^N,
\end{array}
\]
and hence
$
\sum_{N=1}^\infty \bbp\left\{\V_{\bar{x}}\cap \bar{\X}_N\ne
\emptyset\right\} < \infty.
$
It follows by Borel-Cantelli Lemma that
   $\bbp\left(\lim\sup_{N\to\infty} E_N\right)=0$. That is for all $N$
large enough the events
   $E^c_N$ happen w.p.1.
   Now for a given $\e>0$ consider the set $\X_\e:=\{x\in
X':d(x,\X)<\e\}$. Since the set $X'\setminus \X_\e$ is compact we can
choose a finite number of points $x_1,...,x_K\in X'\setminus \X_\e$ and
their respective neighborhoods $\V_1,...,\V_K$ covering the set
$X'\setminus \X_\e$   such that for all $N$ large enough  the events
$\{\V_k\cap \bar{\X}_N=\emptyset\}$, $k=1,...,K$, happen w.p.1. It
follows that w.p.1 for all $N$ large enough   $\bar{\X}_N$ is a subset
of $\X_\e$. This shows that $\cd_N$  tends to zero w.p.1.


To show that $\bbd(\hat{\s}_N,\s^*)\to 0$ w.p.1 the arguments now basically are deterministic, i.e., $\cd_N$ and    $\hat{x}_N\in \hat{\s}_N$ are viewed as random variables, $\cd_N=\cd_N(\w)$, $\hat{x}_N =\hat{x}_N(\w)$,  defined on the  probability space $(\Omega,\F,\bbp)$,  and we want to show that $d(\hat{x}_N(\w),\s^*)$ tends to zero for all $\w\in \Omega$ except on a set of $\bbp$-measure zero. Therefore we consider sequences $\cd_N$ and  $\hat{x}_N$ as deterministic, for a particular (fixed)  $\w\in \Omega$,  and drop mentioning ``w.p.1".
Since $\cd_N\to 0$, there is
  $\tilde{x}_N\in \X$   such that $\|\hat{x}_N-\tilde{x}_N\|$ tends to zero. Note that as an intersection of closed sets,  the set $\X$ is closed. By the assumption (iv) and continuity of $\hat{\Phi}(\cdot,\xi)$
we have that $\hat{\phi}_N(\cdot)$ converges w.p.1 to $\phi(\cdot)$ uniformly on the compact set  $\X\cap X'$ (this is the so-called uniform Law of Large Numbers, e.g., \cite[Theorem 7.48]{SDR}), i.e., for all $\w\in \Omega$ except on a set of $\bbp$-measure zero
\[
\sup_{x\in \X\cap X'}\|\hat{\phi}_N(x)-\phi(x)\|\to 0,\;\;{\rm as}\;N\to\infty.
\]
By passing to a subsequence if necessary   we can assume that  $\hat{x}_N$ converges to a point $x^*$. It follows that $\tilde{x}_N\to x^*$ and hence
$\hat{\phi}_N(\tilde{x}_N)\to \phi(x^*)$. Thus  $\hat{\phi}_N(\hat{x}_N)\to \phi(x^*)$. Since $\Gamma_1$ is closed it follows that $0\in \phi(x^*)+\Gamma_1(x^*)$, i.e., $x^*\in \s^*$. This completes the proof of part (a), and  also implies that the set $\s^*$ is nonempty.

Before proceeding to proof of part (b) we need the following lemma.
\begin{lemma}
Under the assumptions  of Theorem {\rm \ref{th-consist}} it follows that   $\R(0)=0$, $\R(\e)$ is  nondecreasing on  $[0, \infty)$
and $\R(\e)>0$ for all  $\e>0$.
\end{lemma}

\begin{proof} We only need to show that $\R(\e)>0$ for all  $\e>0$, the other two properties are immediate.
 Note that since the set $\s^*$ is nonempty and $\s^*\subset \X\cap X'$, it follows that  the set  $\X\cap X'$ is nonempty.
  Assume for a contradiction that $\R(\bar{\e})=0$ for some $\bar{\e}>0$. Since  $X'$ is compact,  there exists a sequence $\{x_k\}\subset \X\cap X'$ converging to a point $\bar{x}$
   such that $d(x_k, \s^*)\geq\bar{\e}$ and
$$
\lim_{k\to\infty} d(0, \phi(x_k) + \Gamma_1(x_k)) = 0.
$$
Since $\Gamma_1$ is closed and $\phi(\cdot)$ is continuous,  it follows that
$
0\in   \phi(\bar{x}) + \Gamma_1(\bar{x}),
$
i.e.,  $\bar{x}\in \s^*$
This contradicts the fact that $d(\bar{x}, \s^*)\ge\bar{\e}$.
This competes the proof.
\end{proof}
Note that it follows   that  $\R^{-1}(t)$ is  nondecreasing on  $[0, \infty)$
and  tends to zero as $t\downarrow 0$.
\\
{\bf Proof of part (b).}
Let $\delta=\R(\e)/4$.  By part  (a) and the uniform Law of Large Numbers, we have
w.p.1 that for $N$ large enough
$$
\sup_{x\in \X\cap X'} \|\phi(x) - \hat{\phi}_N(x)\|\leq \delta.
$$
Then w.p.1 for  $N$   large enough such that  $\cd_N\leq \e$,
for any point $x\in\bar{\X}_N\cap X'$ with $d(z_x, \s^*)\ge\e$ it follows that
$$
\begin{array}{lll}
&&\hspace{-8mm} d(0, \hat{\phi}_N(x)+\Gamma_1(x)) \\
&\ge&  d(0, \hat{\phi}_N(z_x) + \Gamma_1(z_x)) - \bbd(\hat{\phi}_N(x) + \Gamma_1(x), \hat{\phi}_N(z_x)+\Gamma_1(z_x))\\
                                                        &\ge&  d(0, \phi(z_x)+\Gamma_1(z_x)) - \bbd(\hat{\phi}_N(z_x) + \Gamma_1(z_x), \phi(z_x)+\Gamma_1(z_x))\\
                                                        && -\bbd(\hat{\phi}_N(x) + \Gamma_1(x), \hat{\phi}_N(z_x) + \Gamma_1(z_x))\\
                                                        &\ge& d(0, \phi(z_x)+\Gamma_1(z_x)) - \|\hat{\phi}_N(z_x), \phi(z_x)\| -\|\hat{\phi}_N(x), \hat{\phi}_N(z_x)\|\\
                                                        &&- \bbd(\Gamma_1(x), \Gamma_1(z_x))\\
                                                        &\ge& 4\delta -\delta -\delta - \delta =\delta,
\end{array}
$$
which implies $x\notin \hat{\s}_N$. Then
$$
d(x, \s^*)\leq \|x - z_x\| + d(z_x, \s^*) \leq \tau + \R^{-1}\left(\sup_{x\in \X\cap X'} \|\phi(x) - \hat{\phi}_N(x)\|\right).
$$
This completes the proof.
 \end{proof}

 In case of  the relatively complete recourse  there is no need for  condition (vi) and the estimate  \eqref{eq:dist} holds with $\tau=0$. It is interesting to consider how strong  condition (vi) is. In the following remark we show that   condition (vi) can also hold
without the assumption of relatively complete recourse  under mild   conditions.
\begin{remark}\label{re:polynormcone}
{\rm
In condition (vi), the third inequality of \eqref{eq:convi1} can be easily verified when $N$ sufficiently large and $\hat{\Phi}(\cdot, \xi)$ is Lipschitz continuous with Lipschitz module  $\kappa_{\hat{\Phi}}(\xi)$ and $\bbe[\kappa_{\hat{\Phi}}(\xi)]<\infty$. In Lemma \ref{l:LIP} and Theorem \ref{t:continuousA1} below, we verify the third inequality of \eqref{eq:convi1} under moderate conditions.

Moreover, in the case when $\Gamma_1(\cdot) := \N_C(\cdot)$ with a nonempty polyhedral convex set $C$, the first and second inequality of \eqref{eq:convi1} holds automatically. Let $\cF=\{F_1, \cdots, F_K\}$ be the family  of all nonempty faces of $C$ and
$$
\mathcal{K} := \{ k: \X\cap X'\cap F_k\neq \emptyset,  k=1, \cdots, K\}.
$$
Then w.p.1 for $N$ sufficiently large,  $\bar{\X}_N\cap X'\cap F_k=\emptyset$ for all $k\notin \mathcal{K}$. Note that for all $k\in\mathcal{K}$, $\bar{\X}_N\cap X'\cap F_k\neq \emptyset$. Moreover, it is important to note that for all $x_1\in {\rm reint}(F_k)$ and $x_2\in F_k$, $k\in\{1, \cdots, K\}$, $\N_C(x_1)\subseteq \N_C(x_2)$. Then for any $x\in \bar{\X}_N\cap X'\setminus \X$, there exists $k\in \mathcal{K}$ such that $x\in {\rm reint}(F_k)$. To see this, we assume for contradiction that $x\in F_k\setminus{\rm reint}(F_k)$ for some $k\in\mathcal{K}$ and there is no $k\in\mathcal{K}$ such that $x\in {\rm reint}(F_k)$. Then there exist some $\bar{k}\in\{1, \cdots, K\}$ such that $x\in {\rm reint}(F_{\bar{k}})$ (if $F_{\bar{k}}$ is singleton, then ${\rm reint}(F_{\bar{k}}) = F_{\bar{k}}$) and $\bar{k}\notin \mathcal{K}$. This contradicts that $\bar{\X}_N\cap X'\cap F_k=\emptyset$ for all $k\notin \mathcal{K}$.

Note that $\bbh\left(\bar{\X}_N\cap X', \X\cap X'\right) \leq \cd_N$ and $\cd_N\to0$ as $N\to\infty$ w.p.1.
Let $z_x = \arg\min_{z\in \X\cap X'\cap F_k} \|z-x\|$. Then $\N_C(x)\subseteq \N_C(z_x)$ and for
$$
\tau_{N} := \max_{k\in\mathcal{K}}\max_{x\in \bar{\X}_N\cap X'\cap F_k}\min_{z\in \X\cap X'\cap F_k} \|z - x\|,
$$
we have that
$\tau_{N}\to0$ as $\cd_N\to0$.
Hence  \eqref{eq:convi1} is verified.
}
\end{remark}

\subsection{Exponential rate of convergence}
\label{sec-expon}

We assume  in this section that  
the set $\s^*$  of solutions of the first stage problem is nonempty, and the set $X$ is {\em compact.}
The last assumption of compactness of $X$ can be relaxed to assuming  that there is a compact subset $X'$ of  $X$ such  w.p.1 $\hat{\s}_N\subset X'$, and to deal with the set $X'$  rather than $X$. For simplicity of notation we assume directly compactness of $X$.

Under Assumption \ref{ass-2} and by Lemma \ref{lm-cont}, we have that  $\hat{\Phi}(x,\xi)$, defined in \eqref{geq-3b}, is continuous in $x\in \X$. However  to investigate the exponential rate of convergence, we    need to verify  Lipschitz continuity of $\hat{\Phi}(\cdot,\xi)$. To this end, we assume the {\em Clarke Differential} (CD) regularity property of the second stage generalized equation \eqref{geq-2}. By $\pi_y\partial_{(x,y)}(\Psi(\bar{x}, \bar{y}, \bar{\xi}))$, we denote the projection of the Clarke generalized Jacobian $\partial_{(x,y)}\Psi(\bar{x}, \bar{y}, \bar{\xi})$ in $\bbr^{m\times n}\times \bbr^{m\times m}$ onto $\bbr^{m\times m}$: the set $\pi_y\partial_{(x,y)}\Psi(\bar{x}, \bar{y}, \bar{\xi})$ consists of matrices $J\in\bbr^{m\times m}$ such that the matrix $(S, J)$ belongs to $\partial_{(x,y)}\Psi(\bar{x}, \bar{y}, \bar{\xi})$ for some $S\in\bbr^{m\times n}$.

\begin{definition}
For  $\bar{\xi}\in\Xi$   a solution $\bar{y}$ of the second stage generalized equation  \eqref{geq-2}   is said to be {\em parametrically CD-regular}, at $x=\bar{x}\in \bar{\X}(\bar{\xi})$,  if for each $J\in \pi_y\partial_{(x, y)} \Psi(\bar{x}, \bar{y}, \bar{\xi})$ the solution $\bar{y}$ of the following  SGE is strongly regular
\begin{equation}\label{eq:affine}
0\in\Psi(\bar{x}, \bar{y}, \bar{\xi}) + J ( y - \bar{y} ) + \Gamma_2(y,\bar{\xi}).
\end{equation}
That is, there exist neighborhoods $\U$ of $\bar{y}$ and $\V$ of $0$ such that for every $\eta\in \V$ the perturbed (partially) linearized SGE of \eqref{eq:affine}
\begin{equation*}
\eta \in \Psi(\bar{x}, \bar{y}, \bar{\xi}) + J ( y - \bar{y} )  + \Gamma_2(y,\bar{\xi})
\end{equation*}
has in $\U$ a unique solution $\hat{y}_{\bar{x}}(\eta)$, and the mapping $\eta\to \hat{y}_{\bar{x}}(\eta): \V\to \U$ is Lipschitz continuous.
\end{definition}

\begin{assu}\label{A:CDxxi}
For  all $\bar{x} \in \X$ and  $\xi\in \Xi$, there exists a unique,  parametrically CD-regular solution $\bar{y} = \hat{y}(\bar{x}, \xi)$ of the second stage generalized equation  \eqref{geq-2}.
\end{assu}

\begin{proposition}\label{p:ylipxxi}
Suppose  Assumption {\rm \ref{A:CDxxi}} holds.
Then   the solution mapping   $\hat{y}(x, \xi)$ of the second stage generalized equation  \eqref{geq-2} is a Lipschitz continuous function of $x\in \X$, with Lipschitz constant $\kappa(\xi)$.
\end{proposition}

The result is implied directly by \cite[Theorem 4]{Iz13} and the compactness of $\X\subseteq  X$. Moreover, note that for any $\bar{x}\in \X$, if  the  generalized equation 
$$
0 \in G_{\bar{x}}(y) := \Psi(\bar{x}, \bar{y}, \bar{\xi}) + J ( y - \bar{y} )  + \Gamma_2(y,\bar{\xi}) \;\inmat{ for which } \; G_{\bar{x}}(\bar{y}) \ni0,
$$
has a locally Lipschitz continuous solution function  at $0$ for $\bar{y}$ with Lipschitz constant $\kappa_{G}(\bar{x},\xi)$.
Then by \cite[Theorem 1.1]{DoRo10}, we have
$$
\kappa_{\bar{x}}(\xi) = \kappa_{G}(\bar{x},\xi)\kappa_{\Psi}(\xi) < \infty
$$
 is a Lipschitz constant of the second stage solution function $\hat{y}(x, \xi)$ at $\bar{x}$.

\begin{assu}\label{A:lipschitz}
The set $\X$ is convex, its interior ${\rm int}(\X)\ne \emptyset$, and
for   all $\xi\in\Xi$ and $\bar{x}\in \X$,
the generalized equation 
$$
0\in G_{\bar{x}}(y) = \Psi(\bar{x}, \bar{y}, \xi) + J ( y - \bar{y} )  + \Gamma_2(y,\xi), \;\inmat{ for which } \; G_{\bar{x}}(\bar{y}) \ni0,
$$
has a locally Lipschitz continuous solution function  at $0$ for $\bar{y}$ with Lipschitz constant $\kappa_{G}(\bar{x},\xi)$ and
there exists a measurable function $\bar{\kappa}_G: \Xi\to \bbr_+$ such that, $\kappa_{G}(x,\xi)\leq\bar{\kappa}_G(\xi)$ and $\bbe[\bar{\kappa}_G(\xi)\kappa_\Psi(\xi)]<\infty$.
\end{assu}

Under Assumption \ref{A:lipschitz}, it can be seen  that $\bbe[\hat{y}(x, \xi)]$ is Lipschitz continuous over $x\in \X$ with Lipschitz constant $\bbe[\bar{\kappa}_G(\xi)\kappa_\Psi(\xi)]$. We  consider then the  first stage \eqref{geq-1} of the SGE
as the generalized equation \eqref{geq-3a} with the respective second stage solution $\hat{y}(x, \xi)$ (recall definition \eqref{geq-3b} of $\hat{\Phi}(x,\xi)$ and $\phi(x)$).

\begin{lemma}\label{l:LIP}
Suppose that  Assumptions {\rm \ref{A:CDxxi}--\ref{A:lipschitz}} hold, $\bbe[\kappa_\Phi(\xi)]< \infty$ and
$$
\bbe\left [\kappa_\Phi(\xi)\bar{\kappa}_G(\xi)\kappa_\Psi(\xi)\right ]< \infty.
$$
Then $\hat{\Phi}(x,\xi)$ and $\phi(x)$ are Lipschitz continuous over $x\in \X$ with respective  Lipschitz module $$\kappa_\Phi(\xi) + \kappa_\Phi(\xi)\bar{\kappa}_G(\xi)\kappa_\Psi(\xi)\; {\rm and}\; \bbe[\kappa_\Phi(\xi)] + \bbe[\kappa_\Phi(\xi)\bar{\kappa}_G(\xi)\kappa_\Psi(\xi)].$$
\end{lemma}

\begin{remark}
{\rm Specifically
we   study  Assumptions \ref{ass-2}--\ref{A:lipschitz} in the framework of  the following SGE:
\begin{eqnarray}
\label{geq-1-1}
&&0\in \bbe[\Phi(x, y(\xi), \xi)]+ \Gamma_1(x),\;x\in X, \\
\label{geq-2-1}
&&0\in\Psi(x, y(\xi), \xi)+\N_{\bbr^m_+}(H(x,  y, \xi)), \;\;  \mbox{ for a.e.}\; \xi\in\Xi,
\end{eqnarray}
where $H(x, y, \xi):\bbr^n\times \bbr^m\times \Xi \to \bbr^m$.
Let
$
h(x, y, \xi)  := \min\{ \Psi(x, y, \xi),  H( x, y, \xi)\}.
$
Then the second stage VI \eqref{geq-2-1} is equivalent to
\begin{equation}\label{vieq}
h(x, y, \xi) =0,  \;\;  \mbox{ for a.e.}\; \xi\in\Xi.
\end{equation}
For $x=\bar{x}$ and $\xi\in \Xi$
let  $\bar{y}$ be  a solution of \eqref{vieq}, and suppose that   each matrix $J\in \pi_y\partial h(\bar{x}, \bar{y}, \xi)$ is nonsingular for a.e. $\xi$. Then by Clarke's Inverse Function Theorem, there exists a Lipschitz continuous solution function $\hat{y}(x, \xi)$ such that $\hat{y}(\bar{x}, \xi) = \bar{y}$ and the Lipschitz constant is bounded by $\|J^{-1}(x, y, \xi)S(x, y, \xi)\|$ for all
$$
(S(x, y, \xi), J(x, y, \xi))^\top \in \pi_{x, y} \partial h(x, y, \xi).
$$
Then Assumption \ref{A:CDxxi} holds. Moreover, if we assume
$$
\bbe\left[\|J^{-1}(x, \hat{y}(x, \xi), \xi)S(x, \hat{y}(x, \xi), \xi)\|\right ]<\infty
$$
for all $x\in \X$, then Assumption \ref{A:lipschitz} holds.
}
\end{remark}

Now we investigate   exponential rate of convergence of the  two-stage  SAA problem \eqref{geq-7}--\eqref{geq-8} by using a uniform Large Deviations Theorem (cf., \cite{SDR, shaxu, Xu10A}).
Let
$$
M^i_x(t):=\bbe\left \{{\rm exp}\big(t[\hat{\Phi}_i(x,\xi)-\phi_i(x)]\big)\right\}
$$
be the moment generating function of the random variable
$\hat{\Phi}_i(x,\xi)-\phi_i(x)$,  $i=1, \dots, n$, and
$$
M_\kappa(t):=\bbe\left \{{\rm exp}\left(t\big[\kappa_\Phi(\xi)+ \kappa_\Phi(\xi)\kappa(\xi) - \bbe[\kappa_\Phi(\xi)+ \kappa_\Phi(\xi)\kappa(\xi)\big]\big] \right)\right\}.
$$

\begin{assu}\label{A:moment}
For every $x\in \X$ and $i=1, \dots, n$, the moment generating functions
$M^i_x(t)$ and $M_\kappa(t)$ have
finite values for all
$t$ in a neighborhood of zero.
\end{assu}

\begin{theorem}\label{t:expge}
Suppose that: {\rm (i)} Assumptions {\rm \ref{ass-1}, \ref{ass-3}}--{\rm \ref{A:moment}} hold,
{\rm (ii)} w.p.1 for  $N$ large enough, $\s^*, \hat{\s}_N$ are nonempty, {\rm (iii)} the multifunctions $\Gamma_1(\cdot)$ and $\Gamma_2(\cdot,\xi)$, $\xi\in \Xi$, are  closed and monotone.
Then the following statements hold.
\begin{itemize}
\item[{\rm (a)}]
For  sufficiently small $\e>0$ there
exist  positive constants 
$\varrho=\varrho(\e)$ and  $\varsigma=\varsigma(\e)$, independent
of $N$,  such that
 \begin{equation}
\label{3a}
\bbp \left \{\sup_{x\in \X}\big \|
\hat{\phi}_N(x)-\phi(x)\big\|\geq\e\right \} \leq
\varrho(\e) e^{-N\varsigma (\e)}.
\end{equation}

\item[{\rm (b)}]
Assume  in addition:
{\rm (iv)} The condition of part {\rm (b)} in Theorem {\rm \ref{th-consist}} holds and w.p.1 for $N$ sufficiently large,
\begin{equation}\label{eq:int1}
\mathcal{S}^*\cap {\rm cl}\big({\rm bd}(\X) \cap {\rm int}(\bar{\X}_N)\big) =\emptyset.
\end{equation}
{\rm (v)}  $\phi(\cdot)$ has the following strong monotonicity property for every $x^*\in \s^*$:
\begin{equation}\label{exp-2}
  (x-x^*)^\top (\phi(x)-\phi(x^*)) \ge g(\|x-x^*\|),\;\forall x\in \X,
\end{equation}
where $g:\bbr_+\to \bbr_+$ is such a function  that function  $\crr(\tau):=g(\tau)/\tau$ is monotonically increasing  for $\tau>0$.

 Then $\mathcal{S}^*=\{x^*\}$ is a singleton and for any sufficiently small $\e>0$, there exists $N$ sufficiently large  such that
\begin{equation}
\label{3b}
\bbp \left\{
\bbd(\hat{\mathcal{S}}_N, \mathcal{S}^*)\geq\e \right\} \leq
 \varrho\left(\crr^{-1}(\e)\right) \exp\left (-N\varsigma \big(\crr^{-1}(\e)\big)\right),
\end{equation}
where $\varrho(\cdot)$ and $\varsigma (\cdot)$ are defined in \eqref{3a}, and $\crr^{-1}(\e):=\inf\{\tau>0:\crr(\tau)\ge \e\}$ is the  inverse of $\crr(\tau)$.
\end{itemize}
\end{theorem}

 \begin{proof}
Part (a).
By Lemma \ref{l:LIP}, because of conditions (i) and (ii) and  compactness of $X$, we have  by \cite[Theorem 7.67]{SDR} that  for every $i\in\{1, \dots, n\}$ and $\e>0$ small enough, there
exist  positive constants 
$\varrho_i=\varrho_i(\e)$ and  $\varsigma_i=\varsigma_i(\e)$, independent
of $N$,  such that
\begin{equation*}
\bbp \left \{\sup_{x\in \X}\big |
(\hat{\phi}_N)_i(x)-\phi_i(x)\big|\geq \e \right \} \leq
\varrho_i(\e) e^{-N\varsigma_i (\e)},
\end{equation*}
and hence  \eqref{3a} follows.

Part (b). By condition (iv) we have that
$\bbd(\s^*, \bar{\X}_N\setminus \X)>0$.
Let $\e$ be sufficiently small such that w.p.1 for $N$ sufficiently large,
$$
\bbd(\s^*, \bar{\X}_N\setminus  \X)\geq 3\e.
$$
Note that since $\X\subseteq \bar{\X}_{N+1} \subseteq\bar{\X}_N$, $\bbd(\s^*, \bar{\X}_N\setminus\X)$ is nondecreasing with $N\to\infty$.

By  Theorem \ref{th-consist}, part (b), w.p.1 for $N$ sufficiently large such that $\tau\leq \e$, we have
$$
\R^{-1}\left( \sup_{x\in \X}\|\hat{\phi}_N(x)-\phi(x)\|\right)\leq \e
$$
and
\begin{equation*}
 \bbd(\hat{\s}_N,\s^*) \le \tau + \R^{-1}\left( \sup_{x\in \X}\|\hat{\phi}_N(x)-\phi(x)\|\right)\leq 2\e.
\end{equation*}
Since by condition (iv), when $N$ sufficiently large w.p.1, for any point $\tilde x\in \bar{\X}_N\setminus \X$, $\bbd(\tilde x, \s^*)\geq 3\e$,  which implies $\hat{\s}_N\subset \X$ and then
\begin{equation}\label{ineqr-2}
 \bbd(\hat{\s}_N,\s^*) \le  \R^{-1}\left( \sup_{x\in \X}\|\hat{\phi}_N(x)-\phi(x)\|\right).
\end{equation}
In order to use \eqref{ineqr-2} to derive an exponential rate of convergence  of the SAA estimators we need an upper bound for $\R^{-1}(t)$, or equivalently a lower bound for $\R(\e)$.
Note that because of the monotonicity assumptions we have that  $\s^*=\{x^*\}$.

For  $x\in \X$ and $z\in \Gamma_1(x)$  we have
 \[
 (x-x^*)^\top(\phi(x)-\phi(x^*))=( x-x^*)^\top(\phi(x)+z -\phi(x^*)-z)\le ( x-x^*)^\top(\phi(x)+z),
 \]
 where the last inequality holds since $-\phi(x^*)\in \Gamma_1(x^*)$ and because of monotonicity of $\Gamma_1$. It follows that
 \[
 (x-x^*)^\top (\phi(x)-\phi(x^*))\le \|x-x^*\|\,\|\phi(x)+z\|,
 \]
and since $z\in \Gamma_1(x)$ was arbitrary that
 \[
 (x-x^*)^\top (\phi(x)-\phi(x^*))\le \|x-x^*\|\,d\big (0,\phi(x)+\Gamma_1(x)\big).
 \]
 Together with \eqref{exp-2} this implies
 \begin{equation*}
  d \big (0,\phi(x)+\Gamma_1(x)\big)\ge \crr(\|x-x^*\|).
\end{equation*}
It follows that $\R(\e)\ge \crr(\e)$, $\e\ge 0$, and hence
\begin{equation*}
 \R^{-1}(t)\le \crr^{-1}(t),
\end{equation*}
where $\crr^{-1}(\cdot)$ is the inverse of function $\crr(\cdot)$.
Then by \eqref{3a}, \eqref{3b} holds.
\end{proof}

Note that  if $g(\tau):=c\, \tau^\alpha$ for some constants $c>0$ and $\alpha>1$,    then $\crr^{-1}(t)=(t/c)^{1/(\alpha-1)}$. In particular for $\alpha=2$, condition \eqref{exp-2} assumes strong monotonicity of $\phi(\cdot)$. Note also that condition (iv) is not needed if the relatively complete recourse condition holds.

It is interesting to consider how strong  condition \eqref{eq:int1} is. Note that when $\mathcal{S}^*\subset {\rm int}(\X)$, condition \eqref{eq:int1} holds. Moreover,
we can also see  from the following simple example that even when $\mathcal{S}^*\cap {\rm bd}(\X)\neq \emptyset$,  condition \eqref{eq:int1} may still hold.
\begin{example}
Consider a two-stage SLCP
\begin{eqnarray*}
&&0\le  \begin{pmatrix}
x_1\\
x_2
\end{pmatrix}\perp  \begin{pmatrix}
1 & 0\\
0 & 1
\end{pmatrix}\begin{pmatrix}
x_1\\
x_2
\end{pmatrix} + \begin{pmatrix}
\bbe[y_1(\xi)]\\
\bbe[y_2(\xi)]
\end{pmatrix} \ge 0, \\
&&0\le \begin{pmatrix}
y_1(\xi)\\
y_2(\xi)
\end{pmatrix}
 \perp  \begin{pmatrix}
\alpha(x_1,\xi) & 0\\
0 &\alpha(x_2,\xi)
\end{pmatrix}\begin{pmatrix}
y_1(\xi)\\
y_2(\xi)
\end{pmatrix}- \begin{pmatrix}
x_1\\
x_2
\end{pmatrix}\ge 0,\;  \mbox{a.e. } \xi\in\Xi,
\end{eqnarray*}
where
$$
\alpha(t, \xi) =
\left\{
\begin{matrix}
\frac{1}{t+\xi+51}, & \mbox{if }\; t+\xi \leq 100,\\
0, & \mbox{otherwise},
\end{matrix}
\right.
$$
and $\xi$ follows uniform distribution in $[-50, 50]$.

By simple calculation, we have that $\mathcal{S}^*=\{(0,0)\}$ and $\X=[0,50]\times [0, 50]$. Moreover, consider an iid samples $\{\xi^j\}_{j=1}^N$ with $\max_{j}\xi^j =49$, $\bar{\X}_N=[0, 51]\times[0,51]$. Let $X=\{x: 0\leq x_1, x_2 \leq 100\}$. It is easy to observe that although $\mathcal{S}^*=\{(0,0)\}$ is at the boundary of $\X\cap X$, condition \eqref{eq:int1} still holds.
\end{example}

\begin{remark}
{\rm
It is also interesting to estimate the required  sample size of the  SAA problem for the two-stage SGE. Similar  to a discussion in \cite[p.410]{shaxu}, if there exists a positive constant $\sigma>0$ such that
\begin{equation}\label{eq:mx}
M^i_x(t)\leq {\rm exp}\{\sigma^2t^2/2\}, \;\; \forall t\in\bbr,\;i=1,...,n,
\end{equation}
then it can be  verified that  $I^i_x(z)\geq \frac{z^2}{2\sigma^2}$ for all $z\in\bbr$, where
$
I^i_x(z) := \sup_{t\in\bbr} \{zt-\log M^i_x(t)\}
$
is the large deviations rate function of random variable $\hat{\Phi}_i(x,\xi)-\phi_i(x)$, $i=1, \cdots, n$. Note that if  $\hat{\Phi}_i(x,\xi)-\phi_i(x)$ is subgaussian random variable, \eqref{eq:mx} holds, $i=1,...,n$. Then it can be  verified that if
\begin{equation*}
N\geq \frac{32n\sigma}{\e^2}\left[ \ln(n(2\Pi+1)) +\ln\left(\frac{1}{\alpha}\right) \right],
\end{equation*}
then
\begin{equation*}
\bbp \left \{\sup_{x\in \X}\big \|
\hat{\phi}_N(x)-\phi(x)\big\|\geq\e\right \} \leq
\alpha,
\end{equation*}
where $\Pi:=\left (O(1)D\bbe[\kappa_\Phi(\xi)+ \kappa_\Phi(\xi)\kappa(\xi)]/ \e\right )^n$ and  $D$ is the diameter of $X$. Consequently it follows by \eqref{ineqr-2} that   if
\begin{equation*}
N\geq \frac{32n\sigma}{(\crr^{-1}(\e))^2}\left[ \ln(n(2\hat{\Pi}+1)) +\ln\left(\frac{1}{\alpha}\right) \right],
\end{equation*}
with $\hat{\Pi}:=\left(O(1)D\bbe[\kappa_\Phi(\xi)+ \kappa_\Phi(\xi)\kappa(\xi)]/ \crr^{-1}(\e)\right)^n$, then  we have
\begin{equation*}
\bbp \left \{
\bbd(\hat{\s}_N, \s^*) \geq\e\right \} \leq
\alpha.
\end{equation*}
}
\end{remark}

In the next  section, we will verify the conditions of Theorems \ref{th-consist} and \ref{t:expge} for the  two-stage SVI-NCP under moderate assumptions. 

\section{Two-stage  SVI-NCP and its SAA problem}
\label{sec-svi}
\setcounter{equation}{0}

In this section, we investigate convergence  properties of the  two-stage SGE \eqref{geq-1}--\eqref{geq-2}   when $\Phi(x, y, \xi)$ and $\Psi(x, y, \xi)$ are continuously differentiable w.r.t. $(x, y)$ for a.e. $\xi\in \Xi$ and $\Gamma_1(x):=\N_{C}(x)$ and $\Gamma_2(y):=\N_{\bbr^m_+}(y)$ with $C\subseteq \bbr^n$ being  a nonempty, polyhedral, convex set.  That is, we consider  the mixed two-stage SVI-NCP
 \begin{eqnarray}
\label{geq-SNCP-1}
&&0\in \bbe[\Phi(x, y(\xi), \xi)] + \N_C(x), \\
\label{geq-SNCP-2}
&&0\leq y(\xi)\perp \Psi(x, y(\xi), \xi) \geq0, \;\;  \mbox{ for a.e.}\; \xi\in\Xi,
\end{eqnarray}
and study  convergence analysis of its SAA problem
\begin{eqnarray}
\label{geq-vi-1-saa}
&&0 \in N^{-1}\sum_{j=1}^N \Phi(x, y(\xi^j), \xi^j) + \N_C(x), \\
\label{geq-vi-2-saa}
&&0\leq y(\xi^j) \perp \Psi(x, y(\xi^j), \xi^j)\geq 0, \;\;  j=1,...,N.
\end{eqnarray}

We first give some required  definitions.
Let $\Y$ be the space of measurable functions $u: \Xi \to \bbr^m$ with finite value of $\int \|u(\xi)\|^2P(d\xi)$ and  $\langle\cdot,\cdot\rangle$ denotes
the scalar product in the Hilbert space
$\bbr^n\times \Y$
equipped with ${\cal L}_2$-norm,
that is, for $x,z\in \bbr^n$ and $y,u\in \Y$,
$$
\langle (x,y), (z,u)\rangle := x^\top z+ \int_\Xi y(\xi)^\top u(\xi) P(d\xi).
$$
Consider mapping $\mathcal{G}:\bbr^n\times \Y\to \bbr^n\times \Y$ defined as
$$
\mathcal{G}(x, y(\cdot)):=\big(\bbe[\Phi(x, y(\xi), \xi)],\Psi(x, y(\cdot), \cdot)\big).
$$
Monotonicity properties  of this mapping are defined in the usual way. In particular
the mapping $\mathcal{G}$ is said to be  strongly monotone if there exists a positive number $\bar{\kappa}$ such that for any $(x, y(\cdot)), (z, u(\cdot)) \in \bbr^n\times \Y$, we have
$$
\left\langle \mathcal{G}(x, y(\cdot)) -\mathcal{G}(z, u(\cdot)), \begin{pmatrix}
x-z\\
y(\cdot) - u(\cdot)
\end{pmatrix} \right\rangle \geq \bar{\kappa}(\|x-z\|^2 +\bbe[ \|y(\xi) - u(\xi)\|^2]) . 
$$

\begin{definition}{\rm (\cite[Definition 12.1]{HaKoSc05})}
The mapping $\mathcal{G}: \bbr^n\times \Y \to \bbr^n\times \Y$ is hemicontinuous on $\bbr^n\times \Y$ if $\mathcal{G}$ is continuous on line segments in $\bbr^n\times \Y$, i.e., for every pair of points $(x, y(\cdot)), (z, u(\cdot))\in \bbr^n\times \Y$, the following function is continuous
$$
t \mapsto \left\langle \mathcal{G}(tx + (1-t)z, ty(\cdot) + (1-t)u(\cdot)), \begin{pmatrix}
x-z\\
y(\cdot) - u(\cdot)
\end{pmatrix} \right\rangle.
$$
\end{definition}

\begin{definition}{\rm (\cite[Definition 12.3 (i)]{HaKoSc05})}
The mapping $\mathcal{G}: \bbr^n\times \Y \to \bbr^n\times \Y$ is coercive if there exists $(x_0, y_0(\cdot))\in \bbr^n\times \Y$ such that
$$
\frac{\left\langle\mathcal{G}(x, y(\cdot)),  \begin{pmatrix}
x-x_0\\
y(\cdot) - y_0(\cdot)
\end{pmatrix}\right\rangle}{\|x-x_0\| +\bbe[ \|y(\xi) - y_0(\xi)\|]} \to \infty \;\;\inmat{as}\;\; \|x\| + \bbe[\|y(\xi)\|]\to\infty \;\;\inmat{and}\; (x, y(\cdot))\in \bbr^n\times \Y.
$$
\end{definition}

Note that the strong monotonicity of $\mathcal{G}$ implies the coerciveness of $\mathcal{G}$, see \cite[Chapter 12]{HaKoSc05}. In section 3.1,
we  consider the properties in the second stage SNCP.

\subsection{Lipschitz properties of the second stage solution mapping}

Strong regularity of VI was investigated in Dontchev and Rockafellar \cite{DoRo96}. 
We apply their results to the second stage SNCP.
Consider a linear VI
\begin{equation}\label{eq:linearVI}
0\in Hz + q + \N_U(z),
\end{equation}
where $U$ is a closed nonempty, polyhedral, convex subset of $\bbr^l$.

\begin{definition} {\rm \cite[Definition 2]{DoRo96}}
The critical face condition is said to hold at  $(q_0, z_0)$ if for any choice of  faces $F_1$ and $F_2$ of the critical cone $\C_0$ with $F_2\subset F_1$,
$$
u\in F_1- F_2, \;\; H^\top u\in (F_1 - F_2)^* \; \Longrightarrow \; u=0,
$$
where critical cone $\C_0 = \C(z_0, v_0) := \{z'\in \T_U(x) :z' \perp  v_0\}$ with $v_0 = Hz_0 +q_0$. 
\end{definition}

\begin{theorem}{\rm \cite[Theorem 2]{DoRo96}}
The linear variational inequality \eqref{eq:linearVI} is strongly regular at $(q_0, z_0)$ if and only if the critical face condition holds at $(q_0, z_0)$, where $z_0$ is the solution of the linear VI:
$
0\in Hz + q_0 + \N_U(z).
$
\end{theorem}

\begin{cor}\label{c:Amatrix} {\rm \cite[Corollary 1]{DoRo96}}
A sufficient condition for strong regularity of the linear variational inequality \eqref{eq:linearVI} at $(q_0, z_0)$ is that $u^\top Hu>0$ for all vectors $u\neq 0$ in the subspace spanned by the critical cone $\C_0$.
\end{cor}

Note that when $H$ is a positive definite matrix, the condition in Corollary \ref{c:Amatrix} holds. Then we apply Corollary \ref{c:Amatrix} to the two-stage SVI-NCP and consider the Clarke generalized Jacobian of $\hat{y}(x, \xi)$. To this end, we introduce some notations: let
$$
\begin{array}{lll}
\alpha(\hat{y}(x, \xi))  = \{i:(\hat{y}(x, \xi))_i>(\Psi(x, \hat{y}(x, \xi), \xi))_i \} \\
\beta(\hat{y}(x, \xi))  = \{i: (\hat{y}(x, \xi))_i=(\Psi(x, \hat{y}(x, \xi), \xi))_i \}\\
\gamma(\hat{y}(x, \xi))  = \{i: (\hat{y}(x, \xi))_i<(\Psi(x, \hat{y}(x, \xi), \xi))_i \},
\end{array}
$$
$
\nabla_x \Psi(x, y, \xi) = \begin{pmatrix}
\nabla_x \Psi_\alpha(x, y, \xi) \\
\nabla_x \Psi_\beta(x, y, \xi) \\
\nabla_x \Psi_\gamma(x, y, \xi)
\end{pmatrix}
$
be the Jacobian of $\Psi(x, y, \xi)$ w.r.t. $x$ for given $y$ and $\xi$ and
$$
\nabla_y\Psi(x, y, \xi) = \begin{pmatrix}
\nabla_y\Psi_{\alpha\alpha}(x, y, \xi) & \nabla_y\Psi_{\alpha\beta}(x, y, \xi) & \nabla_y\Psi_{\alpha\gamma}(x, y, \xi)\\
\nabla_y\Psi_{\beta\alpha}(x, y, \xi) &  \nabla_y\Psi_{\beta\beta}(x, y, \xi) & \nabla_y\Psi_{\beta\gamma}(x, y, \xi)\\
\nabla_y\Psi_{\gamma\alpha}(x, y, \xi) & \nabla_y\Psi_{\gamma\beta}(x, y, \xi) & \nabla_y\Psi_{\gamma\gamma}(x, y, \xi)
\end{pmatrix}
$$
be the Jacobian of $\Psi(x, y, \xi)$ w.r.t. $y$ for given $x$ and $\xi$, where the submatrix\linebreak  $\nabla_x\Psi_{\alpha}(x, y, \xi)$ is a matrix with elements $\partial \Psi_{i}(x, y, \xi)/\partial x_j$, $i\in \alpha$, $j\in\{1, \cdots,n\}$ and the submatrix $\nabla_y\Psi_{\alpha\alpha}(x, y, \xi)$ is a matrix with elements $\partial \Psi_{i}(x, y, \xi)/\partial y_j$, $i,j\in \alpha$.

\begin{assu}\label{A:strongly-monotoney}
For a.e. $\xi\in\Xi$ and all $x\in \X\cap C$,
$
\Psi(x, \cdot, \xi)
$
 is strongly monotone, that is there exists a positive valued measurable $\kappa_y(\xi)$ such that for all $y,u\in \bbr^m$,
$$
\left\langle \Psi(x, y, \xi) -\Psi(x, u, \xi), y - u \right\rangle \geq \kappa_y(\xi)\| y -  u\|^2
$$
with $\bbe[\kappa_y(\xi)] < +\infty$.
\end{assu}

Applying Corollary 2.1 in \cite{Ky90} to the second stage of the  SVI-NCP, we have the following lemma.

\begin{lemma}\label{L:yFdifferentiable}
Suppose Assumption  \ref{A:strongly-monotoney} holds and for a fixed $\bar{\xi}\in\Xi$,  $\Psi(x, y, \xi)$ is continuously differentiable w.r.t. $(x, y)$.   
Then  for the fixed $\bar{\xi}\in\Xi$, (a) $\hat{y}(x, \bar{\xi})$ is an unique solution of the second stage NCP \eqref{geq-SNCP-2}, (b) $\hat{y}(x, \bar{\xi})$ is F-differentiable at $\bar{x}\in\X\cap C$ if and only if $\beta(\hat{y}(\bar{x}, \bar{\xi}))$ is empty and
\begin{equation*}
(\nabla_x\hat{y}(\bar{x}, \xi))_\alpha = - (\nabla_y\Psi_{\alpha\alpha}(\bar{x}, \hat{y}(\bar{x}, \xi), \xi))^{-1}\nabla_x\Psi_\alpha(\bar{x}, \hat{y}(\bar{x}, \xi), \xi), \;\; (\nabla_x\hat{y}(\bar{x}, \xi))_\gamma = 0
\end{equation*}
or
\begin{equation*}
\nabla_x\Psi_\beta(\bar{x}, \hat{y}(\bar{x}, \xi), \xi) = \nabla_y\Psi_{\beta\alpha}(\bar{x}, \hat{y}(\bar{x}, \xi), \xi)(\nabla_y\Psi_{\alpha\alpha}(\bar{x}, \hat{y}(\bar{x}, \xi), \xi))^{-1}\nabla_x\Psi_\alpha(\bar{x}, \hat{y}(\bar{x}, \xi), \xi)
\end{equation*}
in this case, the F-derivative of $\hat{y}(\cdot, \xi)$ at $\bar{x}$ is given by
\begin{eqnarray*}
&(\nabla_x\hat{y}(\bar{x}, \xi))_\alpha = - (\nabla_y\Psi_{\alpha\alpha}(\bar{x}, \hat{y}(\bar{x}, \xi), \xi))^{-1}\nabla_x\Psi_\alpha(\bar{x}, \hat{y}(\bar{x}, \xi), \xi), \\
&(\nabla_x\hat{y}(\bar{x}, \xi))_\beta = 0, \;\; (\nabla_x\hat{y}(\bar{x}, \xi))_\gamma = 0.
\end{eqnarray*}
\end{lemma}

\begin{theorem}\label{t:yLipunique}
Let $\Psi: \bbr^n\times \bbr^m\times \Xi\to\bbr^m$ be Lipschitz continuous and continuously differentiable over $\bbr^n\times \bbr^m $ for a.e. $\xi\in\Xi$.
Suppose Assumption \ref{A:strongly-monotoney} holds and $\Phi(x, y, \xi)$ is continuously differentiable w.r.t. $(x, y)$ for a.e. $\xi\in\Xi$. Then for a.e. $\xi\in\Xi$ and $x\in \X$, the following holds.
\begin{itemize}
\item [{\rm (a)}]
 The second stage SNCP \eqref{geq-SNCP-2} has a unique solution $\hat{y}(x, \xi)$ which is parametrically CD-regular  and the mapping $x \mapsto \hat{y}(x, \xi)$ is Lipschitz continuous over $\X\cap X'$, 
where $X'$ is a compact subset of $\bbr^n$.

\item [{\rm (b)}] The Clarke Jacobian of $\hat{y}(x, \xi)$ w.r.t. $x$ is as follows
$$
\hspace{-0.23in}
\begin{array}{lll}
&\partial \hat{y}(x,\xi)= \inmat{conv}\left\{\displaystyle{\lim_{z\to x}} \nabla_z \hat{y}(z,\xi) :  \nabla_z \hat{y}(z, \xi)\right.\\
&\left. \;\;\;\;\;\;\;= -[I - D_{\alpha(\hat{y}(z,\xi))}(I-M(z, \hat{y}(z,\xi), \xi))]^{-1}D_{\alpha(\hat{y}(z,\xi))}L(z, \hat{y}(z,\xi), \xi)
\right\}\\
&\subseteq \inmat{conv}\{-U_J(M(x, \hat{y}(x,\xi),\xi))L(x, \hat{y}(x,\xi),\xi): J\in {\cal J} \},
\end{array}
$$
where $M(x, y, \xi) = \nabla_y \Psi(x, y, \xi)$, $L(x, \hat{y}(x, \xi), \xi) = \nabla_x \Psi(x, \hat{y}(x, \xi), \xi)$, ${\cal J}:=2^{\{1,...,m\}}$, $D_J$ and   $U_J$ are defined in \eqref{eq:Dj} and \eqref{eq:Uj} respectively. 
\end{itemize}
\end{theorem}

\begin{proof}
Part (a). Note that by Lemma  \ref{L:yFdifferentiable} (a), for almost all $\bar{\xi}\in \Xi$ and  every $\bar{x}\in \X\cap X'$, there exists a unique solution $\hat{y}(\bar{x}, \bar{\xi})$ of the second stage SNCP \eqref{geq-SNCP-2}. Moreover, consider the LCP
\begin{equation}\label{eq:linearVIcor}
0\leq y \perp\Psi(\bar{x}, \bar{y}, \bar{\xi}) + \nabla_y\Psi(\bar{x}, \bar{y}, \bar{\xi})(\bar{y} - y)\geq0,
\end{equation}
where $\bar{y} = \hat{y}(\bar{x}, \bar{\xi})$.
By the strong monotonicity of $\Psi(\bar{x}, \cdot, \bar{\xi})$, $\nabla_y\Psi(\bar{x}, \bar{y}, \bar{\xi})$ is positive definite. Then by Corollary \ref{c:Amatrix}, the LCP \eqref{eq:linearVIcor} is strongly regular at $\bar{y}$. This implies the parametrically CD-regular of the second stage SNCP \eqref{geq-SNCP-2} with $\bar{x}$ at solution $\bar{y}$.  Then the Lipschitz property follows from \cite[Theorem 4]{Iz13} and the compactness of $X'$.

Part (b).  For any fixed $\bar{\xi}$, by Part (a), there exists a unique Lipschitz function $\hat{y}(\cdot, \bar{\xi})$ such that $\hat{y}(x, \bar{\xi})$ over $\X$ which solves
\begin{equation*}
0\leq  y\perp \Psi(x, y, \bar{\xi})\geq0.
\end{equation*}

Note that $\hat{y}(\cdot, \bar{\xi})$ is Lipschitz continuous and hence F-differentiable almost everywhere over $\B_\delta(\bar{x})$. Then for any $x'\in \B_\delta(\bar{x})$ such that $\hat{y}(x', \bar{\xi})$ is F-differentiable, by Lemma \ref{L:yFdifferentiable} (b), we have $\beta(\hat{y}(x', \xi))$ is empty and
\begin{equation}\label{eq:yF1}
(\nabla_x\hat{y}(x', \xi))_\alpha = - (\nabla_y\Psi(x', \hat{y}(x', \xi), \xi))^{-1}_{\alpha\alpha}(\nabla_x\Psi(x', \hat{y}(x', \xi), \xi))_\alpha, \;\;  (\nabla_x\hat{y}(x', \xi))_\gamma = 0
\end{equation}
or $\beta(\hat{y}(x', \xi))$ is not empty and
\begin{equation}\label{eq:yF2}
\begin{array}{lll}
&(\nabla_x\hat{y}(x', \xi))_\alpha = - (\nabla_y\Psi(x', \hat{y}(x', \xi), \xi))^{-1}_{\alpha\alpha}(\nabla_x\Psi(x', \hat{y}(x', \xi), \xi))_\alpha, \\
&(\nabla_x\hat{y}(x', \xi))_\beta = 0, \;\; (\nabla_x\hat{y}(x', \xi))_\gamma = 0.
\end{array}
\end{equation}
Let $D_{J} \in \D$ be an $m$-dimensional diagonal matrix with $J\in\mathcal{J}$ and
\begin{equation}\label{eq:Dj}
(D_{J})_{jj} : = \left\{
\begin{array}{ll}
1, & {\rm if } \;  j\in J,\\
0, & {\rm otherwise },
\end{array}
\right.
\end{equation}
$M(x, y, \xi) = \nabla_y \Psi(x, y, \xi) \;\inmat{ and } \; W(x, \xi) = [I - D_{\alpha(\hat{y}(x,\xi))} (I- M(x, y, \xi))]^{-1}D_{\alpha(\hat{y}(x,\xi))}.$
Then by \eqref{eq:yF1} and \eqref{eq:yF2},
\begin{equation*}
\nabla_x \hat{y}(x',\xi) =  -[I - D_{\alpha(\hat{y}(x,\bar{\xi}))}(I-M(x', \hat{y}(x',\bar{\xi}), \xi))]^{-1}D_{\alpha(\hat{y}(x,\bar{\xi}))}L(x', \hat{y}(x',\bar{\xi}), \bar{\xi}),
\end{equation*}
where  $L(x, \hat{y}(x, \xi), \xi) = \nabla_x \Psi(x, \hat{y}(x, \xi), \xi)$. Let
\begin{equation}\label{eq:Uj}
 U_J(M)= (I - D_J (I- M))^{-1}D_J, \;\;\forall J\in \mathcal{J}.
\end{equation}
  By the definition and upper semicontinuity of Clarke generalized Jacobian, we have
 \begin{equation*}
 \begin{array}{lll}
\partial \hat{y}(x,\xi)&=& \inmat{conv}\left\{\displaystyle{\lim_{z\to x}} \nabla_z \hat{y}(z,\xi) : \nabla_z \hat{y}(z,\xi)=\right.  \\
 && \left.-[I - D_{\alpha(\hat{y}(z,\xi))}(I-M(z, \hat{y}(z,\xi), \xi))]^{-1}D_{\alpha(\hat{y}(z,\xi))}L(z, \hat{y}(z,\xi), \xi)
\right\}\\
&\subseteq& \inmat{conv}\{-U_J(M(x, \hat{y}(x,\xi),\xi))L(x, \hat{y}(x,\xi),\xi): J\in {\cal J} \}.
\end{array}
\end{equation*}
We complete the proof.
\end{proof}

Under Assumption \ref{A:strongly-monotoney}, the two-stage SVI-NCP can be reformulated as a single stage SVI with $\hat{\Phi}(x, \xi) = \Phi(x, \hat{y}(x, \xi), \xi)$ and $\phi(x) = \bbe[\hat{\Phi}(x, \xi)]$ as follows
\begin{equation}\label{eq:NCPPhi}
0 \in \phi(x) +  \N_C(x).
\end{equation}
With the results in Theorem \ref{t:yLipunique},  SVI \eqref{eq:NCPPhi} has the following properties. Let 
$$
\Theta(x, y(\xi), \xi)=
\begin{pmatrix}
\Phi(x, y(\xi), \xi)\\
\Psi(x, y(\xi), \xi)
\end{pmatrix}
$$
and $\nabla\Theta(x, y, \xi)$ be the Jacobian of $\Theta$. Then
$$
\nabla\Theta(x, y, \xi) =\begin{pmatrix}
A(x, y, \xi) & B(x, y, \xi)\\
L(x, y, \xi) & M(x, y, \xi)
\end{pmatrix},
$$
where $A(x, y, \xi) = \nabla_x \Phi(x, y, \xi)$, $B(x, y, \xi) = \nabla_y \Phi(x, y, \xi)$, $L(x, y, \xi) = \nabla_x \Psi(x, y, \xi)$ and $M(x, y, \xi) = \nabla_y \Psi(x, y, \xi)$.

\begin{theorem}\label{t:continuousA1}
Suppose the conditions of Theorem \ref{t:yLipunique} hold. Let $X'\subseteq C$ be a compact set, for any $\xi\in \Xi$, $Y(\xi)=\{\hat{y}(x, \xi): x\in X'\}$ and $\nabla\Theta(x, y, \xi)$ be the Jacobian of $\Theta$. Assume
\begin{equation}\label{eq:lessinfty}
\bbe[\|A(x, \hat{y}(x, \xi), \xi) - B(x, \hat{y}(x, \xi), \xi) M(x, \hat{y}(x, \xi), \xi)^{-1}L(x, \hat{y}(x, \xi), \xi)\|]<+\infty
\end{equation}
over $\X\cap X'$.
Then
\begin{itemize}
\item [{\rm (a)}]
$\hat{\Phi}(x, \xi)$ is Lipschitz continuous w.r.t. $x$ over $\X\cap X'$  for all $\xi\in \Xi$.

\item [{\rm (b)}] $\bbe[\hat{\Phi}(x, \xi)]$ is Lipschitz continuous w.r.t. $x$ over $\X\cap X'$.
\end{itemize}
\end{theorem}
\begin{proof}
Part (a). By the compactness of $X'$ and Theorem \ref{t:yLipunique} (a), $Y(\xi)$ is compact for almost all $\xi\in\Xi$. By the continuity of $\nabla\Theta(x, \hat{y}(x, \xi), \xi)$,
we have
$$
A(x, \hat{y}(x, \xi), \xi) - B(x, \hat{y}(x, \xi), \xi) M(x, \hat{y}(x, \xi), \xi)^{-1}L(x, \hat{y}(x, \xi), \xi)
$$
is continuous over $X'$. Then we have
$$
\sup_{x\in X'}\|A(x, \hat{y}(x, \xi), \xi) - B(x, \hat{y}(x, \xi), \xi) M(x, \hat{y}(x, \xi), \xi)^{-1}L(x, \hat{y}(x, \xi), \xi)\|<+\infty.
$$
Moreover,  by Theorem \ref{t:yLipunique} (b), the Lipschitz module of $\hat{\Phi}(x, \xi)$, denote by ${\rm lip}_{\Phi}(\xi)$ satisfies
$$
\begin{array}{lll}
&&{\rm lip}_{\Phi}(\xi)\\
& \leq& \displaystyle{\sup_{x\in X'}}\|A(x, \hat{y}(x, \xi), \xi) - B(x, \hat{y}(x, \xi), \xi) M(x, \hat{y}(x, \xi), \xi)^{-1}L(x, \hat{y}(x, \xi), \xi)\|\\
&<&+\infty.
\end{array}
$$

Part (b). it comes from  Part (a) and \eqref{eq:lessinfty} directly.
\end{proof}

\subsection{Existence, uniqueness and CD-regularity of the solutions}

Consider the mixed SVI-NCP \eqref{geq-SNCP-1}-\eqref{geq-SNCP-2} and its one stage reformulation \eqref{eq:NCPPhi}. If we replace Assumption \ref{A:strongly-monotoney} by the following assumption, we can have stronger results.

\begin{assu}\label{A:strongly-monotone}
For a.e.  $\xi\in\Xi$, $\Theta(x, y(\xi), \xi)$
 is strongly monotone with parameter $\kappa(\xi)$ at  $(x, y(\cdot))\in C\times \Y$, where 
$\bbe[\kappa(\xi)] < +\infty$.

\end{assu}
Note that Assumption \ref{A:strongly-monotoney} can be implied by Assumption \ref{A:strongly-monotone} over $C\times \Y$.

\begin{theorem}\label{t:existenceunique}
Suppose Assumption \ref{A:strongly-monotone}  holds  over $C \times \Y$ and $\Phi(x, y, \xi)$ and $\Psi(x, y, \xi)$ are continuously differentiable w.r.t. $(x, y)$ for a.e. $\xi\in \Xi$. Then
\begin{itemize}
\item[{\rm (a)}]
 $\mathcal{G}: C\times \Y \to C\times \Y$  is strongly monotone and hemicontinuous.

\item[{\rm (b)}]
 For all $x$ and almost all $\xi\in\Xi$, $\Psi(x, y(\xi), \xi)$ is strongly monotone and continuous w.r.t. $y(\xi)\in\bbr^m$.

\item [{\rm (c)}]
 The  two-stage SVI-NCP \eqref{geq-SNCP-1}-\eqref{geq-SNCP-2} has a unique solution.

\item [{\rm (d)}] The  two-stage SVI-NCP \eqref{geq-SNCP-1}-\eqref{geq-SNCP-2} has relatively complete recourse, that is
 for all $x$ and almost all $\xi\in\Xi$, the  NCP \eqref{geq-SNCP-2} has a unique solution.

 \end{itemize}
\end{theorem}
\begin{proof}
Parts (a) and (b) come from Assumption \ref{A:strongly-monotone} over $C\times \Y$ directly. Since the strong monotonicity of $\mathcal{G}$  and $\Psi$ implies the coerciveness of $\mathcal{G}$ and $\Psi$, see \cite[Chapter 12]{HaKoSc05}, by \cite[Theorem 12.2 and Lemma 12.2]{HaKoSc05}, we have Part (c) and Part (d).
\end{proof}

With the results in sections 3.1 and above, we have the following theorem  by only assume that Assumption \ref{A:strongly-monotone} holds in a neighborhood of {\rm Sol}$^*\cap X'\times \Y$.

\begin{theorem}\label{t:continuous}
Let 
{\rm Sol}$^*$ be the solution set of the mixed SVI-NCP \eqref{geq-SNCP-1}-\eqref{geq-SNCP-2}. Suppose (i) there exists a compact set $X'$ such that  ${\rm Sol}^*\cap X'\times \Y$ is nonempty,  (ii) Assumption \ref{A:strongly-monotone} holds over {\rm Sol}$^*\cap X'\times \Y$  and (iii) the conditions of Theorem \ref{t:continuousA1} hold.
Then
\begin{itemize}

\item [{\rm (a)}] For any $(x, y(\cdot))\in {\rm Sol}^*$, every matrix in $\partial \hat{\Phi}(x)$ is positive definite and   $\hat{\Phi}$ and $\phi$ are strongly  monotone at $x$.

\item [{\rm (b)}]  Any solution $x^*\in \mathcal{S}^*\cap X'$ of SVI \eqref{eq:NCPPhi}   is CD-regular and an isolate solution.

\item [{\rm (c)}] Moreover, if replacing conditions (i) and (ii) by supposing (iv) Assumption~\ref{A:strongly-monotone} holds over $\bbr^n\times \Y$, then SVI \eqref{eq:NCPPhi} has a unique solution $x^*$ and the solution is CD-regular.
\end{itemize}
\end{theorem}

\begin{proof}
Part (a). Note that under Assumption \ref{A:strongly-monotone}, for any $(x, y(\cdot))\in {\rm Sol}^*$, the matrix
$$
\begin{pmatrix}
A(x, y(\xi), \xi) & B(x, y(\xi), \xi)\\
L(x, y(\xi), \xi) & M(x, y(\xi), \xi)
\end{pmatrix} \succ 0.
$$
From (ii) of Lemma 2.1 in \cite{ChSuXu17}, we have
$$
A(x, y(\xi), \xi) - B(x, y(\xi), \xi)U_J(M(x, y(\xi),\xi))L(x, y(\xi),\xi)\succ 0, \,\,\, \forall J\in {\cal J}.
$$
 For any $\bar{x}$ such that $(\bar{x}, \bar{y}(\cdot))\in {\rm Sol}^*$,
let
$\B_\delta(\bar{x})$ be a small neighborhood of $\bar{x}$,
$$
\D_{\hat{y}}(\bar{x}):=\{x': x'\in \B_\delta(\bar{x}), \;  \hat{y}(x', \xi)  \inmat{ is F-differentiable w.r.t. $x$ at }x'\}
$$
and
$$
\D_{\hat{\Phi}}(\bar{x}):=\{x': x'\in \B_\delta(\bar{x}), \;  \hat{\Phi}(x', \xi)  \inmat{ is F-differentiable w.r.t. $x$ at }x'\}.
$$
Since $\Phi(x, y, \xi)$ is continuously differentiable w.r.t. $(x, y)$, $\hat{y}(\cdot, \xi)$ is F-differentiable w.r.t. $x$, which implies $\hat{\Phi}(\cdot, \xi)$ is F-differentiable w.r.t. $x$. Then $\D_{\hat{y}}(\bar{x})\subseteq \D_{\hat{\Phi}}(\bar{x})$.
Moreover, since $\hat{y}(x, \xi)$ and $\hat{\Phi}(x, \xi)$ are Lipschitz continuous w.r.t. $x$ over $\B_\delta(\bar{x})$, they are F-differentiable almost everywhere over $\B_\delta(\bar{x})$. Then the measure of $\D_{\hat{\Phi}}(\bar{x})\backslash \D_{\hat{y}}(\bar{x})$ is zero.
By Theorem \ref{t:yLipunique} (b) and the definition of Clarke generalized Jacobian, we have
\begin{equation}\label{eq:clarkehatphi}
\begin{array}{lll}
&\partial_x \hat{\Phi}(\bar{x}, \xi)\\
 =& \inmat{conv}\left\{ \displaystyle{\lim_{x'\to \bar{x}}} \nabla_x \hat{\Phi}(x',  \xi):  x'\in \D_{\hat{\Phi}}(\bar{x})\right\}\\
=& \inmat{conv}\left\{ \displaystyle{\lim_{x'\to \bar{x}}} \nabla_x \Phi(x', \hat{y}(x', \xi), \xi) + \nabla_y \Phi(x', \hat{y}(x', \xi), \xi) \nabla_x \hat{y}(x', \xi): x'\in \D_{\hat{y}}(\bar{x}) \right\}\\
=& \inmat{conv}\left\{\displaystyle{\lim_{x'\to \bar{x}}} A(x', \hat{y}(x', \xi), \xi) \right.\\
&\left.- B(x', \hat{y}(x', \xi), \xi)U_{\alpha(\hat{y}(x',\xi))}(M(x', \hat{y}(x',\xi),\xi))L(x', \hat{y}(x',\xi),\xi):\right.   \left. x'\in \D_{\hat{y}}(\bar{x}) \right\}\\
\subset&\inmat{conv}\left\{ A(x, \hat{y}(x, \xi), \xi) \right.\\
&\left.- B(x, \hat{y}(x, \xi), \xi)U_J(M(x, \hat{y}(x,\xi),\xi))L(x, \hat{y}(x,\xi),\xi): J\in {\cal J} \right\},
\end{array}
\end{equation}
where the second equation is from \cite[Theorem 4]{Wa81} and the fact that the measure of $\D_{\hat{\Phi}}(\bar{x})\backslash \D_{\hat{y}}(\bar{x})$ is $0$.
By \eqref{eq:clarkehatphi}, every matrix in $\partial_x \hat{\Phi}(\bar{x}, \xi)$ is positive definite. And then $\hat{\Phi}$ is strongly  monotone which implies $\phi$ is strongly monotone at $\bar{x}$.

Part (b). By Corollary \ref{c:Amatrix}, the linearized SVI
\begin{equation*}
0 \in  V_{x^*}(x-x^*) + \bbe[\hat{\Phi}(x^*,  \xi)] + \N_C(x),
\end{equation*}
is strongly regular for all $V_{x^*}\in \partial \phi(x^*) \subseteq \bbe[\partial_x \hat{\Phi}(x^*, \xi)]$. Then the NCP \eqref{eq:NCPPhi} at $x^*$ is CD-regular. Moreover,  by the definition of CD regular, $x^*$ is a unique solution of the  NCP \eqref{eq:NCPPhi} over a neighborhood of $x^*$.

Part (c). By Part (a) and Theorem \ref{t:existenceunique}, NCP \eqref{eq:NCPPhi} has a unique solution $x^*$. The CD regular of NCP \eqref{eq:NCPPhi} at $x^*$ follows from Part (b).
\end{proof}

\subsection{Convergence analysis of the SAA  two-stage SVI-NCP}

Consider the two-stage SVI-SNCP  \eqref{geq-SNCP-1}-\eqref{geq-SNCP-2} and its
SAA problem \eqref{geq-vi-1-saa}-\eqref{geq-vi-2-saa}.

We discuss the existence and uniqueness of the solutions of SAA two-stage SVI \eqref{geq-vi-1-saa}-\eqref{geq-vi-2-saa} under Assumption \ref{A:strongly-monotone} over $C\times \Y$ firstly.
 Define
$$
\mathcal{G}_N := \begin{pmatrix}
N^{-1}\sum_{j=1}^N \Phi(x, y(\xi^j), \xi^j)\\
\Psi(x, y(\xi^1), \xi^1)\\
\vdots\\
\Psi(x, y(\xi^N), \xi^N)
\end{pmatrix}.
$$

\begin{theorem}\label{t:existenceN}
Suppose Assumption \ref{A:strongly-monotone}  holds  over $C \times \Y$ and $\Phi(x, y, \xi)$ and $\Psi(x, y, \xi)$ are continuously differentiable w.r.t. $(x, y)$ for a.e. $\xi\in \Xi$. Then
\begin{itemize}
\item[{\rm (a)}]
 $\mathcal{G}_N: C\times \Y \to C\times \Y$ which is strongly monotone with $N^{-1}\sum_{j=1}^N\kappa(\xi^j)$ and hemicontinuous.

\item [{\rm (b)}]
 The SAA two-stage SVI \eqref{geq-vi-1-saa}-\eqref{geq-vi-2-saa} has a unique solution.

 \end{itemize}
\end{theorem}

\begin{proof}
By Assumption \ref{A:strongly-monotone}, we have Parts (a) and (b).
\end{proof}

Then we investigate the almost sure convergence and convergence rate of  the first stage solution $\bar{x}_N$ of \eqref{geq-vi-1-saa}-\eqref{geq-vi-2-saa}  to optimal solutions of the true problem by only supposing Assumption \ref{A:strongly-monotone} holds at a neighborhood of Sol$^*\cap X'\times \Y$.

Note that the normal cone multifunction $x\mapsto \N_C(x)$ is closed.
Note also that function
$
\hat{\Phi}(x,\xi)=\Phi(x, \hat{y}(x,\xi),\xi),
$
where $\hat{y}(x,\xi)$ is a solution  of the second stage problem \eqref{geq-SNCP-2}. Then the first stage of  SAA problem with second stage solution  can be written as
\begin{equation}\label{eq:NCPPhiSAA}
0\in N^{-1}\sum_{j=1}^N \hat{\Phi}(x,\xi^j) + \N_C(x).
\end{equation}

Under the conditions (i)-(iii) of Theorem \ref{t:continuous},
the two-stage SVI-SNCP  \eqref{geq-SNCP-1}-\eqref{geq-SNCP-2} and its
SAA problem \eqref{geq-vi-1-saa}-\eqref{geq-vi-2-saa} satisfy
conditions  of Theorem \ref{th-consist} and    with $\R^{-1}(t)\leq \frac{t}{c}$ for some positive number $c$ (by Remark \ref{re:polynormcone},  the strongly monotone of $\phi$ and the argument in the proof of Part (b), Theorem \ref{t:expge} ). Then
Theorem \ref{th-consist} can be applied directly.

\begin{definition}\label{D:strongstable}{\rm \cite{FaPa03,PaSuSu03}}
A solution $x^*$ of the SVI \eqref{eq:NCPPhi} is said to be strongly stable if for every open neighborhood $\V$ of $x^*$ such that ${\rm SOL}(C, \phi)\cap {\rm cl}\V = \{x^*\}$, there exist two positive scalars ${\delta}$ and $\epsilon$ such that for every continuous function $\tilde \phi$ satisfying
$$
\sup_{x\in C\cap {\rm cl}\V} \|\tilde \phi(x) - \phi(x)\| \leq \epsilon,
$$
the set ${\rm SOL}(C, \tilde \phi)\cap \V$ is a singleton; moreover, for another continuous function $\bar{\phi}$ satisfying the same condition as $\tilde \phi$, it holds that
\begin{equation*}
\|x-x'\| \leq {\delta} \|[ \phi(x) - \tilde \phi(x) ] - [ \phi(x') - \bar{\phi}(x') ]\|,
\end{equation*}
where $x$ and $x'$ are  elements in the sets ${\rm SOL}(C, \tilde \phi)\cap \V$ and ${\rm SOL}(C, \bar{\phi})\cap \V$, respectively.
\end{definition}

\begin{theorem}\label{t:convergenceVI}
Suppose conditions (i)-(iii) of Theorem \ref{t:continuous} hold.  Let $x^*$ be a solution of   the SVI \eqref{eq:NCPPhi} 
and $X'$ be a compact set such that $x^*\in {\rm int}(X')$. Assume there exists $\e>0$ such that for $N$ sufficiently large,
\begin{equation}\label{eq:int}
x^*\notin {\rm cl}( {\rm bd}(\X) \cap {\rm int}( \bar{\X}_N\cap X')).
\end{equation}
Then there exist a solution $\hat{x}_N$ of the SAA problem \eqref{eq:NCPPhiSAA} and
a positive scalar ${\delta}$ such that $\|\hat{x}_N - x^*\|\to0$ as $N\to\infty$ w.p.1 and
for $N$ sufficiently large w.p.1
\begin{equation}\label{eq:xphi}
\|\hat{x}_N - x^*\| \leq {\delta} \sup_{x\in \X\cap X'} \| \hat{\phi}_N(x) - \phi(x)  \|.
\end{equation}
\end{theorem}
\begin{proof}
By Theorem \ref{t:continuous} (b), the SVI \eqref{eq:NCPPhi} at $x^*$ is CD-regular. By \cite[Theorem 3]{PaSuSu03} and \cite{FaPa03}, $x^*$ is a strong stable solution of the SVI \eqref{eq:NCPPhi}.  Note that by Theorem \ref{t:continuous} (a) and \cite[Theorem 7.48]{SDR}, we have
$$
 \sup_{x\in \X\cap X'}\|\hat{\phi}_N(x) - \phi(x)\|
$$ converges to $0$ uniformly.
Then by Definition \ref{D:strongstable} and \eqref{eq:int},
there exist two positive scalars ${\delta}$, $\epsilon$  such that for  $N$ sufficiently large, w.p.1
$$
\sup_{x\in \X\cap X'} \| \hat{\phi}_N(x) - \phi(x)  \| \leq \min\{\epsilon, \e/\delta\}
$$
and
$$
\|\hat{x}_N - x^*\| \leq {\delta} \sup_{x\in \X\cap X'} \| \hat{\phi}_N(x) - \phi(x)  \|,
$$
which implies $\hat{x}_N\in \X$.
\end{proof}

Note that Theorem \ref{t:convergenceVI} guarantees that $\R^{-1}(t)\leq {\delta}t$ and condition \eqref{eq:int} is discussed after Theorem \ref{t:expge}. Note also that replacing conditions (i) - (ii) and condition \eqref{eq:int} by supposing condition (iv) of Theorem \ref{t:continuous},  conclusion  \eqref{eq:xphi} also holds. Moreover, in this case,
by Theorem \ref{t:continuous}~(c) and Theorem \ref{t:existenceN},  $x^*$ and $\hat{x}_N$ are the  unique solutions of the SVI \eqref{eq:NCPPhi} and its SAA problem \eqref{eq:NCPPhiSAA} respectively.

Then we consider the exponential rate of convergence. Note that under Assumption \ref{A:strongly-monotoney}, for SAA problem of mixed two-stage SVI-NCP \eqref{geq-vi-1-saa}-\eqref{geq-vi-2-saa}, Assumptions {\rm \ref{ass-1}, \ref{A:CDxxi}}, {\rm \ref{A:lipschitz}} and condition (iii) in Theorem \ref{t:expge} hold. If we replace Assumption \ref{A:strongly-monotoney} by Assumption \ref{A:strongly-monotone} over Sol$^*\cap X'\times \Y$, we have the following theorem.

\begin{theorem} Let $X'\subset C$ be a convex compact subset such that $\B_\delta(x^*)\subset X'$.
Suppose  the conditions in Theorem \ref{t:convergenceVI}  and  Assumption \ref{A:moment} hold. Then  for any $\e>0$ there
exist  positive constants ${\delta}>0$ (independent
of $\e$), $\varrho=\varrho(\e)$ and  $\varsigma=\varsigma(\e)$, independent
of $N$,  such that
 \begin{equation}
\label{3a1}
\prob \left \{\sup_{x\in \X}\big \|
\hat{\phi}_N(x)-\phi(x)\big\|\geq\e\right \} \leq
\varrho(\e) e^{-N\varsigma (\e)},
\end{equation}
and
\begin{equation}
\label{3b1}
\prob \left \{
\|x_N-x^*\|\geq\e\right \} \leq
 \varrho(\e/{\delta}) e^{-N\varsigma (\e/{\delta})}.
\end{equation}
\end{theorem}
\begin{proof}
By Theorem \ref{t:continuous} (a), Assumption \ref{A:moment} and \cite[Theorem 7.67]{SDR}, the conditions of Theorem \ref{t:expge} (a) hold and then  \eqref{3a1} holds.
Under condition \eqref{eq:int} in Theorem \ref{t:convergenceVI},
\eqref{3b1} follows from  \eqref{eq:xphi} and \eqref{3a1}.
\end{proof}

\section{Examples}
\label{sec-examp}
\setcounter{equation}{0}
In this section, we illustrate our theoretical results in the last sections by a two-stage stochastic
 non-cooperative game of two players \cite{ ChSuXu17, PSS17}. Let $\xi:\Omega\to\Xi\subseteq \bbr^{d}$ be a random vector,  $x_i\in \bbr^{n_i}$ and $y_i(\cdot)\in \Y_i$ be the strategy vectors and policies of the $i$th player at the first stage and second stage, respectively, where  $\Y_i$ is a measurable function space from $\Xi$ to $\bbr^{m_i}$, $i=1, 2$, $n=n_1+n_2$, $m=m_1+m_2$. In this two-stage stochastic game, the $i$th player solves the following  optimization problem:
\begin{equation}\label{eq:nex-1}
\min_{x_i\in [a_i, b_i]} \; \theta_i(x_i, x_{-i}) + \bbe[\psi_i(x_i, x_{-i}, y_{-i}(\xi), \xi)],
\end{equation}
where
$
\theta_i(x_i, x_{-i}) : = \frac{1}{2}x_i^T H_{i}x_i + q_i^Tx_i +  x_i^T P_{i}x_{-i},
$
\begin{equation}\label{eq:nex-2}
\psi_i(x_i, x_{-i}, y_{-i}(\xi), \xi) : = \min_{y_i\in [l_i(\xi), u_i(\xi)]} \; \phi_i(y_i,x_i, x_{-i}, y_{-i}(\xi), \xi)
\end{equation}
is the optimal value function of the recourse action $y_i$ at the second stage with
$$
\phi_i(y_i, x_i, x_{-i}, y_{-i}(\xi), \xi) = \frac{1}{2}y_i^\top Q_{i}(\xi)y_i + c_i(\xi)^\top y_i + \sum_{j=1}^2 y_i^\top S_{ij}(\xi)x_j  + y_i^\top O_{i}(\xi)y_{-i}(\xi),
$$
$a_i, b_i\in \bbr^{n_i}$, $l_i, u_i: \Xi\to \bbr^{m_i}$ are vector valued measurable functions, $l_i(\xi) < u_i(\xi)$ for all $\xi\in \Xi$,
 $H_i$ and $Q_i(\xi)$  are symmetric positive definite matrices for a.e $\xi\in\Xi$, $x=(x_1, x_2)$, $y(\cdot) = (y_1(\cdot), y_2(\cdot))$, $x_{-i}=x_{i'}$ and $y_{-i}=y_{i'}$,  for  $ {i'\neq i}$.
We use $y_i(\xi)$ to denote the unique solution of \eqref{eq:nex-2}.

By \cite[Theorem 5.3 and Corollary 5.4]{GaDu82}, $\psi_i(x_i, x_{-i}, y_{-i}(\xi), \xi)$ is continuously differentiable w.r.t. $x_i$ and
$$
\nabla_{x_i} \psi_i(x_i, x_{-i}, y_{-i}(\xi), \xi)
= S_{ii}^T(\xi)y_i(\xi).
$$
Hence the two-stage stochastic game can be formulated as a two-stage linear SVI
\begin{equation*}
\begin{array}{rcll}
-\nabla_{x_i} \theta_i(x_i, x_{-i}) - \bbe[\nabla_{x_i}\psi_i(x_i, x_{-i}, y_{-i}(\xi), \xi)] & \in & \N_{[a_i, b_i]}(x),  \\
-\nabla_{y_i(\xi)} \phi_i(y_i(\xi), x_i, x_{-i}, y_{-i}(\xi), \xi) & \in & \N_{[l_i(\xi), u_i(\xi)]}(y_i(\xi)),  \\
&&\inmat{for  a.e. } \xi\in \Xi,
\end{array}
\end{equation*}
for $i=1, 2,$  with the following matrix-vector form
\begin{equation}\label{eq:sboxvi}
\begin{array}{rcl}
-A x - \bbe[B(\xi)y(\xi)] - h_1& \in & \N_{[a, b]}(x)\\
-M(\xi) y(\xi) - L(\xi) x - h_2(\xi) & \in & \N_{[l(\xi), u(\xi)]}(y(\xi)), \;\; \inmat{ for  a.e. } \xi\in \Xi,
\end{array}
\end{equation}
where
$$
A =
\begin{pmatrix}
H_1 & P_{1} \\
P_{2} & H_2
\end{pmatrix}, \;\;
B(\xi) =
\begin{pmatrix}
S^T_{11}(\xi) & 0 \\
0 & S^T_{22}(\xi)
\end{pmatrix},
$$
$$
L(\xi) = \begin{pmatrix}
S_{11}(\xi) & S_{12}(\xi) \\
S_{21}(\xi) & S_{22}(\xi)
\end{pmatrix}, \;\;
M(\xi) = \begin{pmatrix}
Q_1(\xi) & O_{1}(\xi) \\
O_{2}(\xi) & Q_2(\xi) \end{pmatrix},
$$
$h_1 = (q_1, q_2)$ and $h_2(\xi) = (c_1(\xi),  c_2(\xi))$.  Moreover, if there exists a
positive continuous function $\kappa(\xi)$ such that $\bbe[\kappa(\xi)]<+\infty$ and for a.e. $\xi\in \Xi$,
\begin{equation}\label{eq:ABmatrix}
\begin{pmatrix}
z^\top, u^\top
\end{pmatrix}
\begin{pmatrix}
A & B(\xi)\\
L(\xi) & M(\xi)
\end{pmatrix}\begin{pmatrix}
z\\
u
\end{pmatrix} \geq \kappa(\xi)(\|z\|^2+\|u\|^2), \; \;\; \forall z\in\bbr^n, \; u\in\bbr^{m},
\end{equation}
the two-stage box constrained SVI \eqref{eq:sboxvi} satisfy Assumption \ref{A:strongly-monotone}. By the Schur complement condition for positive definiteness \cite{Horn}, a sufficient condition for \eqref{eq:ABmatrix} is
$$
4H_2 - (P_1+P_2^\top)H_1^{-1}(P_1+P_2^\top) \quad \mbox{ is positive definite}
$$
and for some $k_1>0$ and a.e. $\xi\in\Xi$,
$$
\lambda_{\rm min}(M(\xi)+M(\xi)^\top -(B(\xi)+L(\xi)^\top)(A+A^\top)^{-1}(B(\xi)+L(\xi)^\top)) \geq k_1 >0,
$$
where $\lambda_{\rm min}(V)$ is the smallest eigenvalue of   $V \in \bbr^{m\times m}$.

Under condition \eqref{eq:ABmatrix},   by Corollary \ref{c:Amatrix} and Theorem \ref{t:existenceunique}, the conditions in Theorem \ref{t:expge} hold for \eqref{eq:sboxvi}. To see this, we only need to show condition (vi) of Theorem \ref{t:expge} holds for \eqref{eq:sboxvi}. Consider the second stage VI of \eqref{eq:sboxvi} for fixed $\xi$ and $x$, by the proof of \cite[Lemma 2.1]{CW12}, we have
$$
\hat{y}(x,\xi) - \hat{y}(x',\xi) = -(I-D(x,x',\xi)+D(x,x',\xi)M(\xi))^{-1}D(x,x',\xi)L(\xi)(x-x'),
$$
which implies
\begin{equation}\label{eq:payex}
\partial_x \hat{y}(x,\xi) \subseteq \{-(I-D+DM(\xi))^{-1}DL(\xi): D\in \mathcal{D}_0\},
\end{equation}
where $D(x,x',\xi)$ is a diagonal matrix with diagonal elements
$$
d_i=\left\{
\begin{array}{lll}
0, \;\;\;  \inmat{ if } (\hat{y}(x,\xi))_i - z_i(x,\xi), (\hat{y}(x',\xi))_i - z_i(x',\xi) \in [u_i(\xi), \infty),\\
0, \;\;\;  \inmat{ if } (\hat{y}(x,\xi))_i - z_i(x,\xi), (\hat{y}(x',\xi))_i - z_i(x',\xi) \in (-\infty, l_i(\xi)],\\
1, \;\;\;  \inmat{ if } (\hat{y}(x,\xi))_i - z_i(x,\xi), (\hat{y}(x',\xi))_i - z_i(x',\xi) \in (l_i(\xi), u_i(\xi)),\\
\frac{(\hat{y}(x,\xi))_i - (\hat{y}(x',\xi))_i}{(\hat{y}(x,\xi))_i - z_i(x,\xi) - ((\hat{y}(x',\xi))_i - z_i(x',\xi)}, \;  \inmat{ otherwise},
\end{array}
\right.
$$
$z_i(x,\xi)=(M(\xi)\hat{y}(x,\xi)+L(\xi)x+h_2(\xi))_i$, $d_i\in[0,1]$, $i=1, \cdots, m$, $\mathcal{D}_0$ is a set of diagonal matrices in $\bbr^{m\times m}$ with the diagonal elements in $[0,1]$.
Then we consider the one stage SVI with $\hat{y}(x,\xi)$ as follows
\begin{equation}\label{eq:sboxvione}
-A x - \bbe[B(\xi)\hat{y}(x, \xi)] - h_1 \in  \N_{[a, b]}(x).
\end{equation}
By using the similar arguments as in the proof of Theorem \ref{t:continuous} and \eqref{eq:payex}, every elements of the Clarke Jacobian of $Ax+\bbe[B(\xi)\hat{y}(x, \xi)]+h_1$ is a positive definite matrix. Then \eqref{eq:sboxvione} is strong monotone and hence condition (vi) of Theorem \ref{t:expge} holds.
In what follows, we verify the convergence results in Theorem \ref{t:expge} numerically.

Let $\{\xi^j\}_{i=1}^N$ be an iid sample of random variable $\xi$. Then the SAA problem of \eqref{eq:sboxvi} is
\begin{equation}\label{eq:sboxvisaa}
\begin{array}{rcl}
-A x - \frac{1}{N}\sum_{j=1}^NB(\xi^j)y(\xi^j) - h_1& \in & \N_{[a, b]}(x)\\
-M(\xi^j) y(\xi^j) - L(\xi^j) x - h_2(\xi^j) & \in & \N_{[l(\xi^j), u(\xi^j)]}(y(\xi^j)), \;\;  j=1, \dots, N.
\end{array}
\end{equation}
PHM converges to a solution of  \eqref{eq:sboxvisaa} if condition (\ref{eq:ABmatrix}) holds.

\begin{alg} [PHM] Choose $r>0$ and  initial points  $x^0 \in \bbr^{n}$, $x^0_j=x^0\in\bbr^n$, $y_j^0\in\bbr^m$ and $w_j^0\in\bbr^n$, $j=1, \cdots, N$ such that $\frac{1}{N}\sum_{j=1}^Nw^0_j=0$.  Let $\nu=0$.

\textbf{Step 1.} For $j=1, \cdots, N$, solve the box constrained VI
\begin{equation}
\label{eq:slcp-two1N-PH}
\begin{array}{rcl}
 - Ax_j - B(\xi^j)y_j -  h_1 - w^{\nu}_{j} - r (x_j - x^{\nu}_j) & \in & \N_{[a,b]}(x_j),  \\
 - M(\xi^j)y_j - L(\xi^j)x_j  - h_2(\xi^j) -r (y_j - y^{\nu}_j) & \in & \N_{[l(\xi^j), u(\xi^j)]}(y_j),
\end{array}
\end{equation}
and obtain a solution $(\hat{x}^{\nu}_j, \hat{y}^{\nu}_j), $ $ j=1, \cdots, N$.

\textbf{Step 2.}
Let $\bar{x}^{\nu+1}= \frac{1}{N}\sum_{j=1}^N \hat{x}^{\nu}_j.$ For $j = 1, \cdots, N$, set
$$
 \;\;\,\,
x^{\nu+1}_j=\bar{x}^{\nu+1},\,\, \,
\;\; y^{\nu+1}_j =  \hat{y}^{\nu}_j, \,\,\,\,\, w_{j}^{\nu+1}  = w_{j}^{\nu} + r (\hat{x}^{\nu}_j - x^{\nu+1}_j). $$
\end{alg}
Note that PHM is well-defined if $\begin{pmatrix} A & B(\xi^j)\\ L(\xi^j) & M(\xi^j)\end{pmatrix}$, $j=1,\cdots, N$ are positive semidefinite, that is, \eqref{eq:slcp-two1N-PH} has a unique solution for each $j$, even for some $x$ and $\xi^j$ the second stage problem
$$
 - M(\xi^j)y - L(\xi^j)x  - h_2(\xi^j)  \in  \N_{[l(\xi^j), u(\xi^j)]}(y)
$$
has no solution.

\subsection{Generation of matrices satisfying condition (\ref{eq:ABmatrix})}
We generate matrices $A$, $B(\xi), L(\xi), M(\xi)$  by the following procedure.
Randomly generate a symmetric positive definite matrix $H_1\in \bbr^{n_1\times n_1}$,
matrices $P_1\in \bbr^{n_1\times n_2}, P_2\in \bbr^{n_2\times n_1}$. Set $H_2=\frac{1}{4}(P_1^\top+P_2)H_1^{-1}(P_1+P_2^\top) +\alpha I_{n_2}$,
where $\alpha$ is a positive number.
Randomly generate matrices with entries within $[-1,1]$:
$$\bar{S}_{11}\in \bbr^{m_1\times n_1}, \quad
\bar{S}_{12}\in \bbr^{m_1\times n_2}, \quad
\bar{S}_{21}\in \bbr^{m_2\times n_1},
$$
$$
\bar{S}_{22}\in \bbr^{m_2\times n_2}, \quad
\bar{O}_{1}\in \bbr^{m_1\times m_2}, \quad
\bar{O}_{2}\in \bbr^{m_2\times m_1}.$$
Randomly generate two symmetric matrices
$\bar{Q}_1\in \bbr^{m_1\times m_1}$ and $\bar{Q}_{2}\in \bbr^{m_2\times m_2}$ whose diagonal entries are greater than $m -1 +\alpha$,  
off-diagonal entries are in $[-1,1]$, respectively.

Generate  an iid sample  $\{\xi^j\}_{j=1}^N\subset [0,1]^{10}\times[-1,1]^{10} $ of random variable $\xi\in \bbr^{20}$ following uniformly distribution over $\Xi=[0,1]^{10}\times[-1,1]^{10}$. Set
$$ S_{11}(\xi)=\xi_1^j\bar{S}_{11}, \,
S_{12}(\xi)=\xi_2^j\bar{S}_{12}, \,
S_{21}(\xi)=\xi_3^j\bar{S}_{21},
$$
$$
S_{22}(\xi)=\xi_4^j\bar{S}_{22},\, O_{1}(\xi)=\xi_5^j\bar{O}_{1}, \, O_{2}(\xi)=\xi_6^j\bar{O}_{2},$$$$ Q_1(\xi)=\bar{Q}_1 + (\xi_7^j +\frac{(n+m)^2}{\lambda_{\rm min}(A+A^T)})I_{m_1} \quad   Q_2(\xi)=\bar{Q}_2+ (\xi_8^j +\frac{(n+m)^2}{\lambda_{\rm min}(A+A^T)} )I_{m_2}.$$
Set $B(\xi^j), L(\xi^j), M(\xi^j)$  as in (\ref{eq:sboxvi}).

The matrices generated by this procedure satisfy condition (\ref{eq:ABmatrix}). Indeed,
since $H_1$ and $
4H_2 - (P_1+P_2^T)H_1^{-1}(P_1+P_2^T)$ are  positive definite,
  by the Schur complement condition for positive definiteness \cite{Horn},  $A+A^T$ is symmetric positive definite, and thus $A$ is positive definite. Moreover, since
   the matrix
$
\bar{M}:=\begin{pmatrix}
\bar{Q}_1 & \bar O_1\\
\bar O_2 &  \bar Q_2
\end{pmatrix}
$
is diagonal dominance with positive diagonal entries $\bar{M}_{ii}\ge m -1 +\alpha$, it is positive definite and the eigenvalues $M+M^T$ are greater than $2\alpha$. Hence,  for any
$y \in \bbr^{m}$, we have
\begin{eqnarray*}
&&y^T(M(\xi)+M(\xi)^T -(B(\xi)^T+L(\xi))(A+A^T)^{-1}(B(\xi)+L(\xi)^T))y\\
&\ge& (2\alpha+\frac{(n+m)^2}{\lambda_{\rm min}(A+A^T)})\|y\|^2- \frac{1}{\lambda_{\rm min}(A+A^T)}\|(B(\xi)^T+L(\xi))\|^2\|y\|^2 \ge 2\alpha \|y\|^2,
\end{eqnarray*}
where we use $\|B(\xi)^T+L(\xi)\|^2\le \|B(\xi)^T+L(\xi)\|^2_1\le (m+n)^2$. Using the Schur complement condition for positive definiteness \cite{Horn} again, we obtain condition (\ref{eq:ABmatrix}).

Finally, we generate the box constraints, $h_1$ and $h_2(\cdot)$. For the first stage,
the lower bound is set as $a=0 {\bf 1}_n,$ and the upper bound of the box constraints  $b$ is randomly generated  from $[1, 50]^6$. For the second stage, we set $l(\xi)=(1+\xi_9)\bar{l}$ and $u(\xi)=(1+\xi_{10})\bar{u}$, where ${\bf 1}_n \in \bbr^n$ is a vector with all elements 1,  $\bar{l}$ is randomly generated  from $[0,1]^{10}$ and  $\bar{u}$ is randomly generated from $[3,50]^{10}$. Moreover, the vector $h_1$ is randomly generated  from $[-5,5]^6$ and $h_2(\xi)=(\xi_{11}, \cdots, \xi_{20})$ is a random vector following uniform distribution over $[-1,1]^{10}$.

\subsection{Numerical results}

 For each sample size of $N=10, 50, 250, 1250, 2250$, we randomly generate   $20$ test problems  and solve the box-constrained VI in Step 1 of PHM by the homotopy-smoothing method \cite{ChenYe1999}. We stop the iteration when
 \begin{equation}\label{eq:res}
 {\bf res}:= \|x- {\rm mid}(x-Ax- \frac{1}{N}\sum^N_{j=1}B(\xi^j)\hat{y}(x,\xi^j) -h_1, a, b)\|\le 10^{-5},
 \end{equation}
 or the iterations reach 5000, where mid$(\cdot)$ denotes the componentwise median operator, $\hat{y}(x,\xi^j)$ is the solution of the second stage box constrained VI with  $x$ and~$\xi^j$.

 Parameters for the numerical tests are chosen as follows:  $n_1=n_2=3, m_1=m_2=5, \alpha=1$ and maximize iteration number is $5000$.

 Figures 1-6 show the convergence tendency of $x_1$, $x_2$, $x_3$, $x_4$, $x_5$ and $x_6$ respectively. Note that since we use the homotopy-smoothing method to solve the box-constrained VI in Step 1 of PHM and the stop criterion is $10^{-5}$,  $x_2$ is not always feasible. However, $[a_i-x_i]_++[x_i-b_i]_+\leq 10^{-5}$, $i=1, \dots, 6$, which is related to the stopping  criterion of the homotopy-smoothing method.

\begin{figure}[!htb]
\begin{minipage}[t]{0.5\textwidth}
\centering
\includegraphics[width=2.5in, height=1.6in]{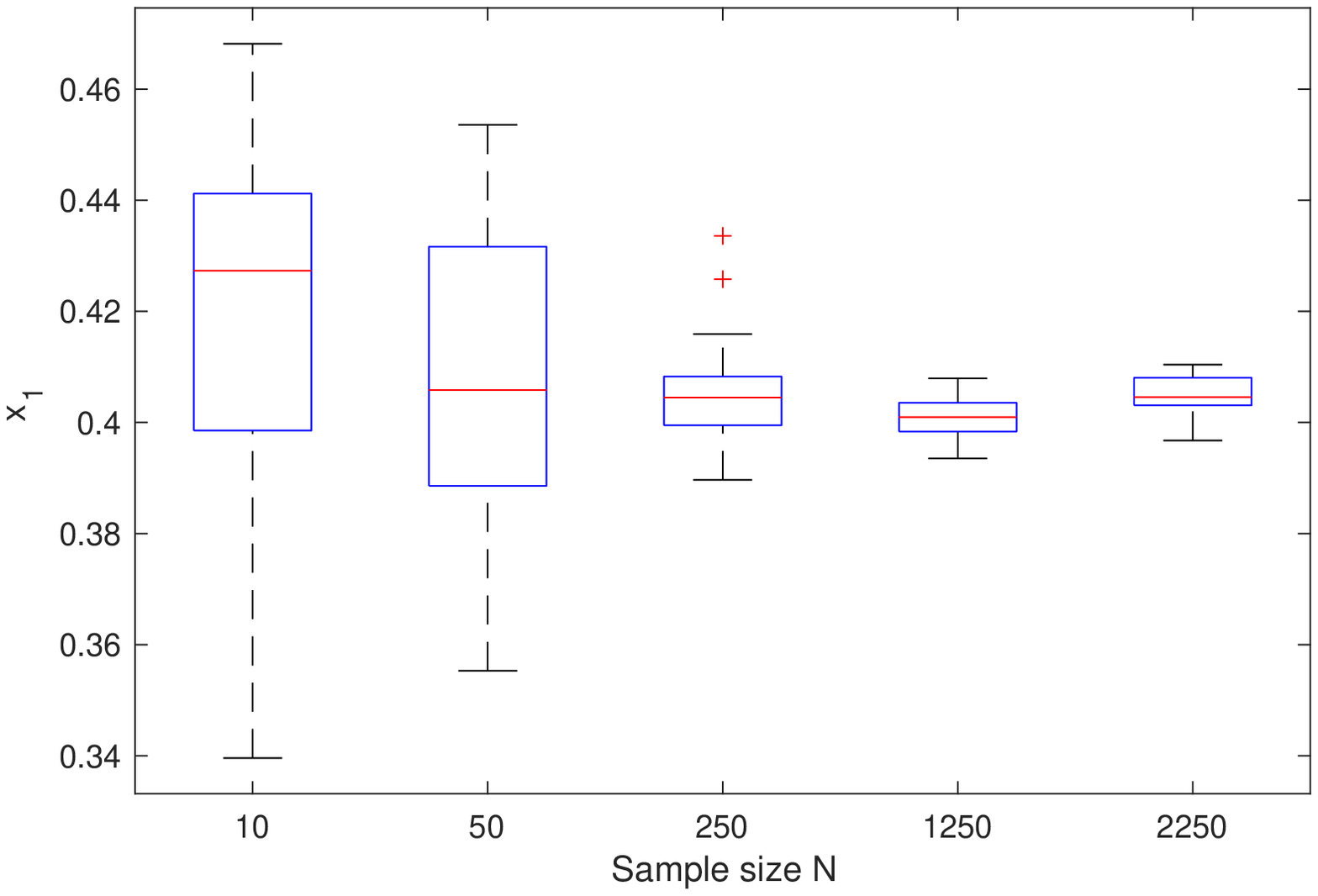}
\caption{Convergence of $x_1$}
\label{fig:side:a}
\end{minipage}%
\begin{minipage}[t]{0.5\textwidth}
\centering
\includegraphics[width=2.5in, height=1.6in]{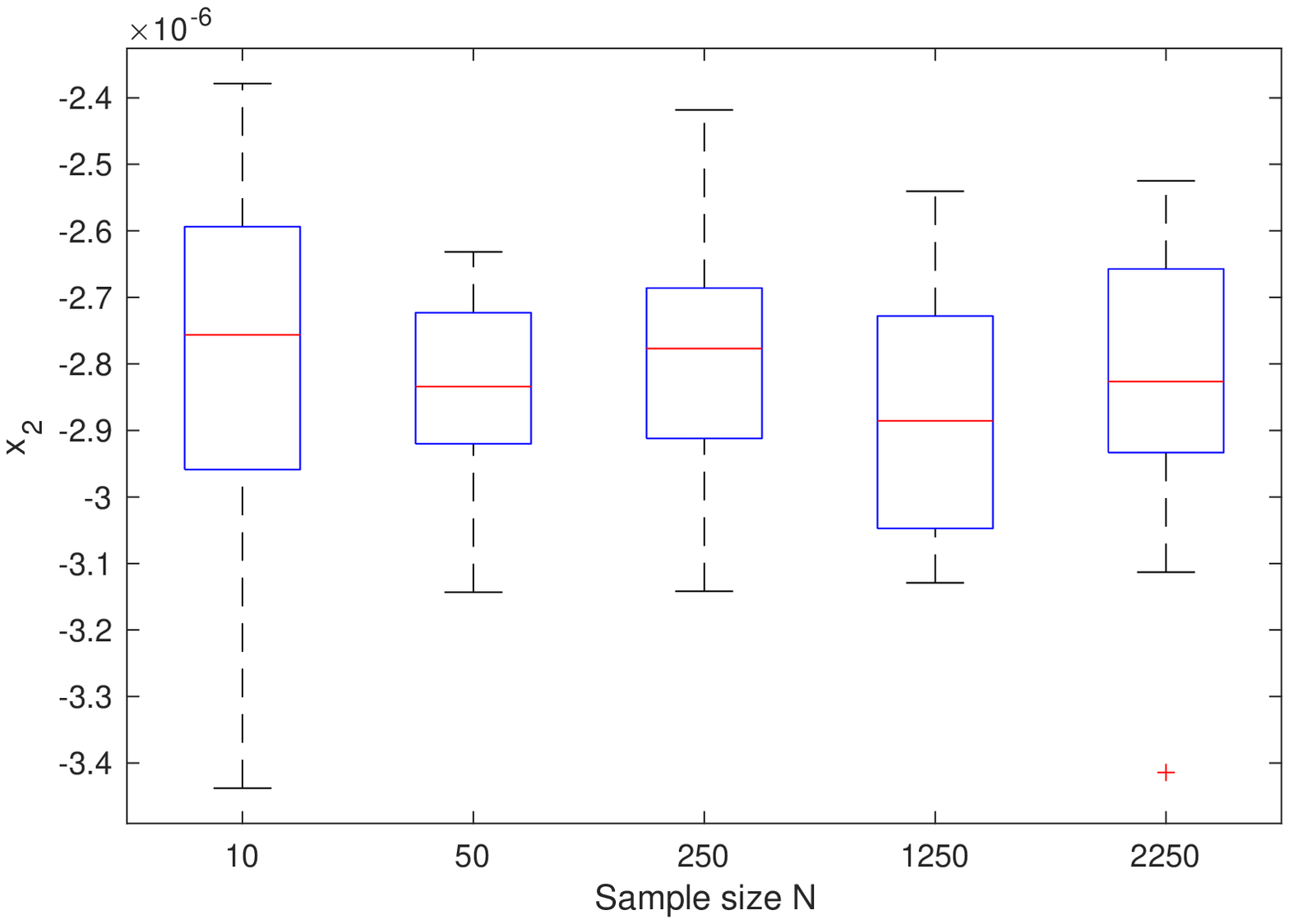}
\caption{Convergence of $x_2$}
\label{fig:side:b}
\end{minipage}
\end{figure}
\begin{figure}[!htb]
\begin{minipage}[t]{0.5\textwidth}
\centering
\includegraphics[width=2.5in, height=1.6in]{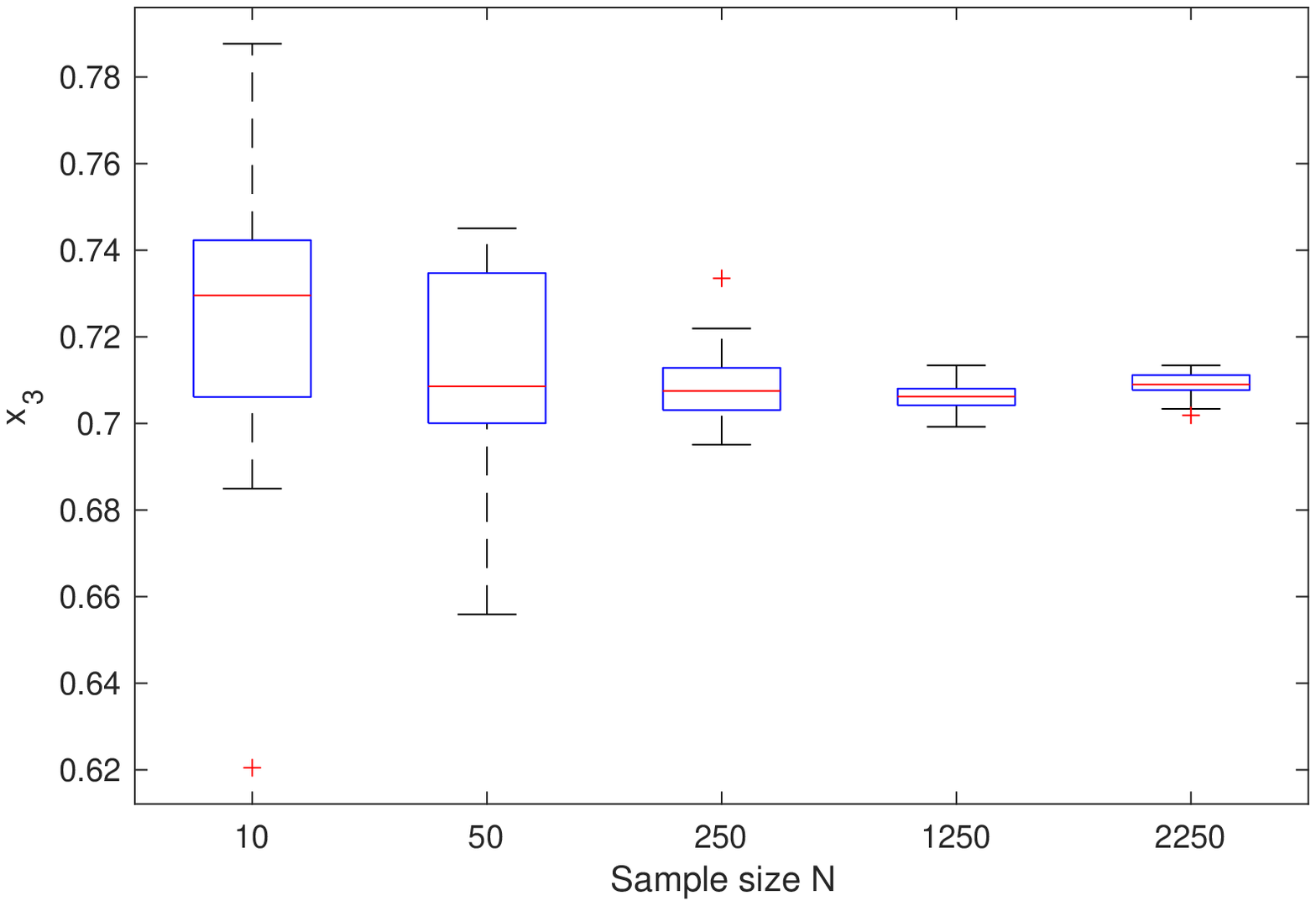}
\caption{Convergence of $x_3$}
\label{fig:side:c}
\end{minipage}%
\begin{minipage}[t]{0.5\textwidth}
\centering
\includegraphics[width=2.5in, height=1.6in]{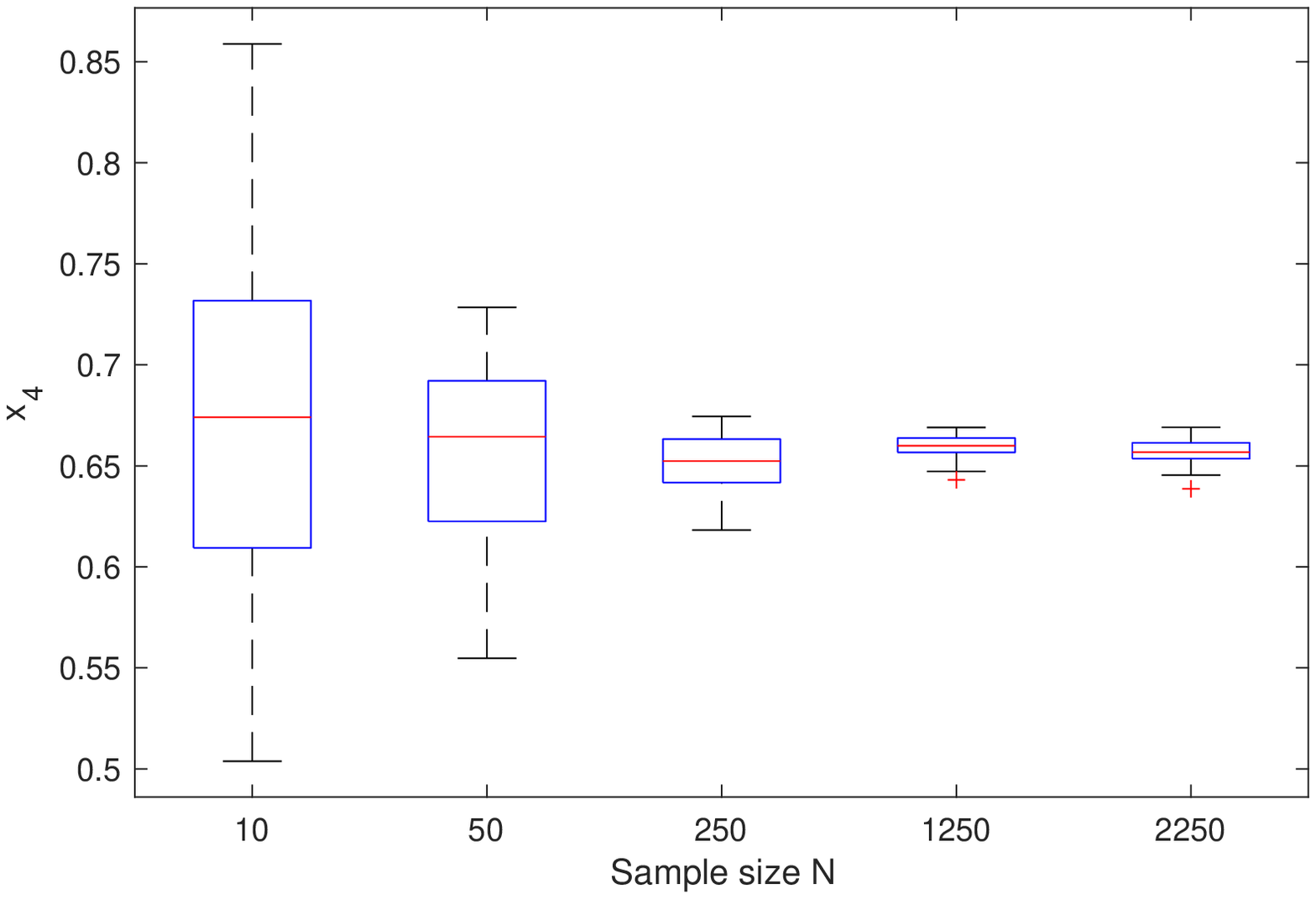}
\caption{Convergence of $x_4$}
\label{fig:side:d}
\end{minipage}
\end{figure}
\begin{figure}[!htb]
\begin{minipage}[t]{0.5\textwidth}
\centering
\includegraphics[width=2.5in, height=1.6in]{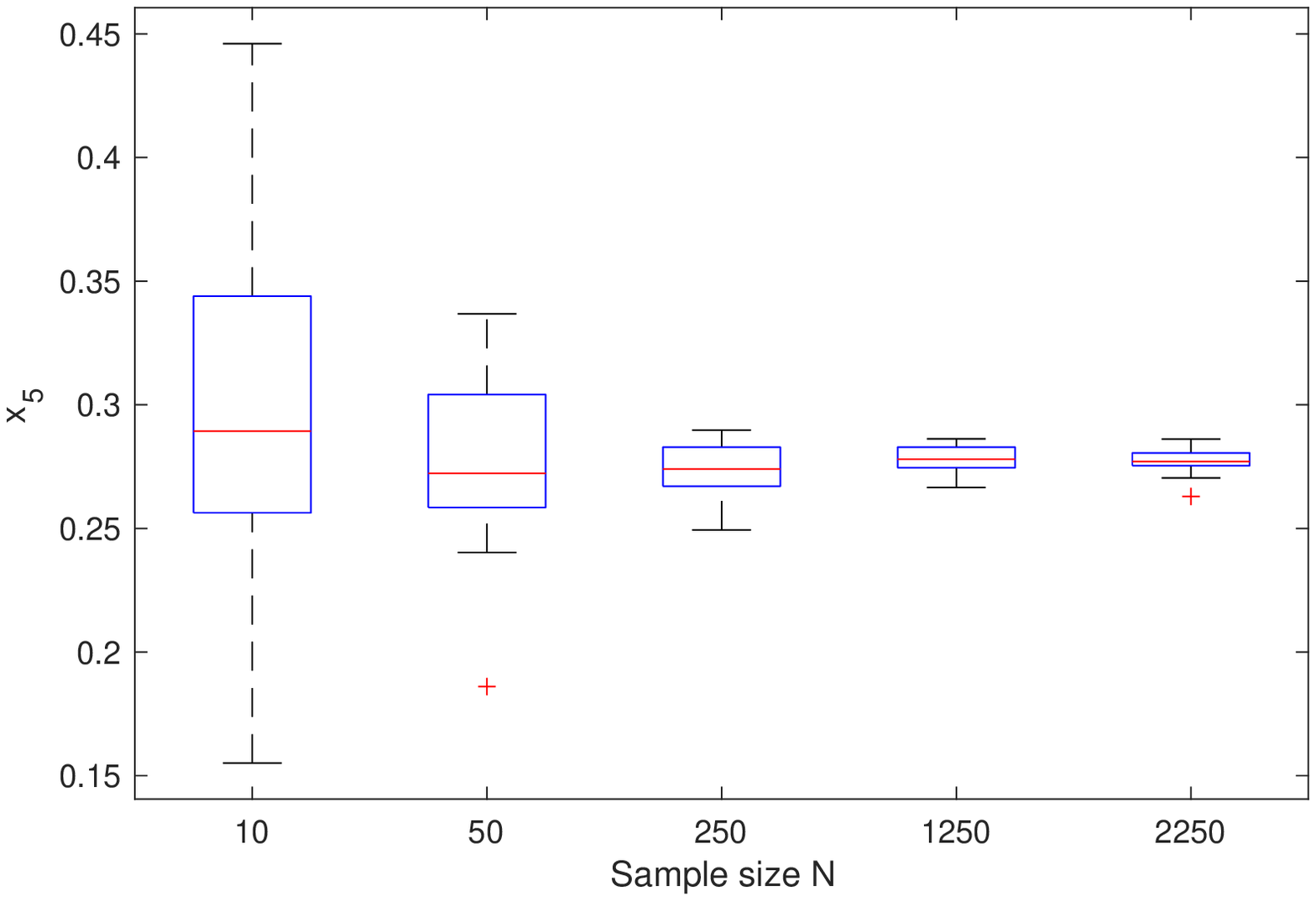}
\caption{Convergence of $x_5$}
\label{fig:side:e}
\end{minipage}%
\begin{minipage}[t]{0.5\textwidth}
\centering
\includegraphics[width=2.5in, height=1.6in]{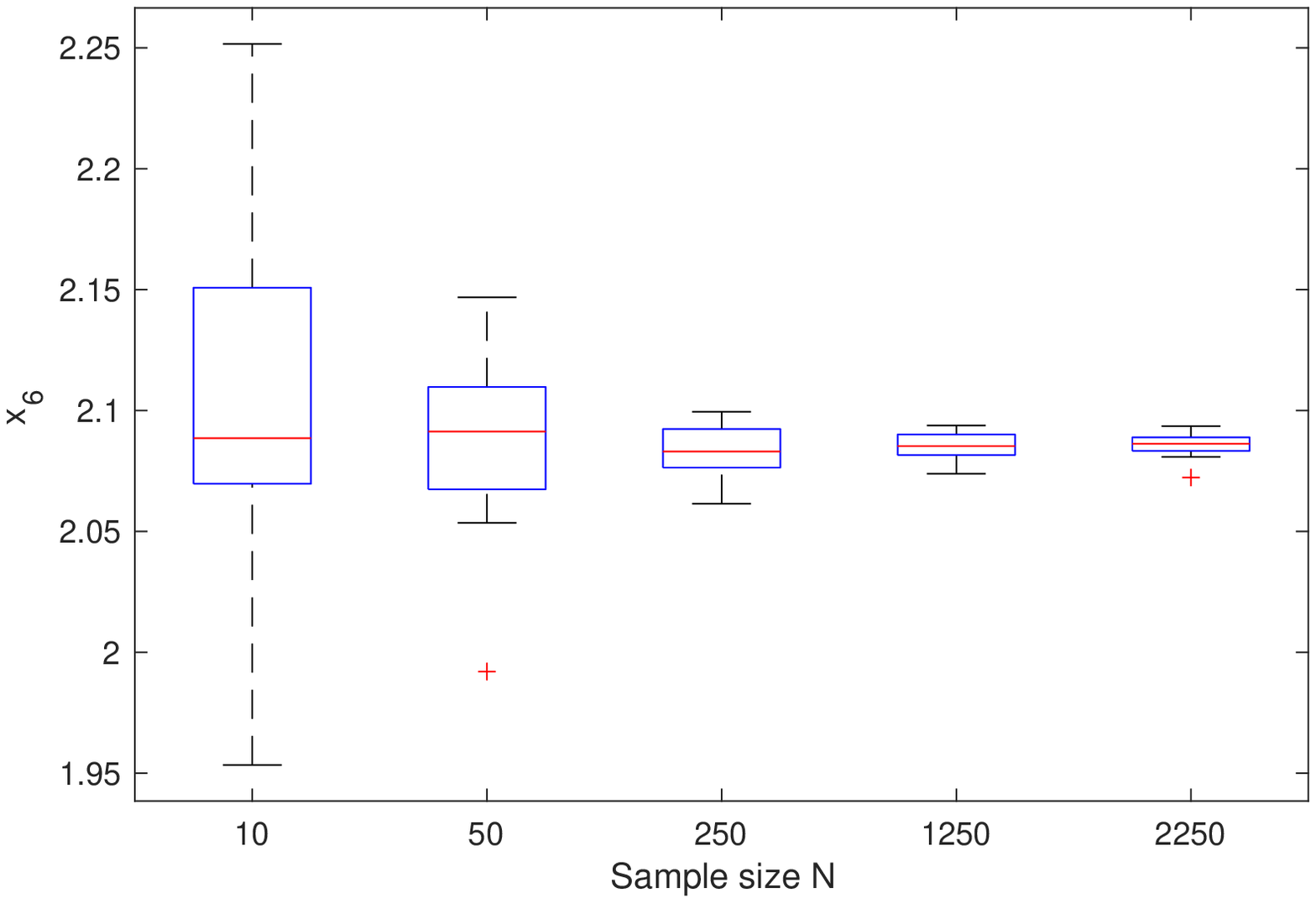}
\caption{Convergence of $x_6$}
\label{fig:side:f}
\end{minipage}
\end{figure}

  We 
  use $x^{N_t,j}$ $j=1,\ldots, 3000, t=1,\ldots, 5$  to denote the computed solutions with sample size $N_t$ for the $j$-th test problem shown in Figure 1.
  Then we computer the mean, variance and  $95\%$ confidence interval (CI) of the  corresponding {\bf res} defined in \eqref{eq:res} with $x=x^{N_t,j}$ by using a new set of 20 randomly generated test problems with sample size $N=3000$ for computing $\hat{y}(x^{N_t,j},\xi^j), j=1,\ldots, 3000, t=1,\ldots, 5$.  We can see that the average of the mean, variance and width of 95\% CI of {\bf res} in Table~1 decrease as the sample size increases.

\begin{table}[h]
\scriptsize
 \centering
\begin{tabular}{ c | c | c | c | c | c }
 \hline
 &  $N_1=10$ & $N_2=50$ & $N_3=250$ & $N_4=1250$ & $N_5=2250$ \\
 \hline
mean  & 0.22449  &  0.13753  &  0.04820  &  0.02885  &  0.02500 \\
  \hline
variance& 0.01984  &  0.00605  &  0.00118  &  0.00023  &  0.00016 \\
\hline
  95\% CI & [0.2158, 0.2332] & [0.1349, 0.1402] & [0.0477, 0.0487]  & [0.0287, 0.0290] & [0.0249, 0.0251] \\
\hline
\end{tabular}
\caption{ Mean, variance and 95\% confidence interval (CI) of {\bf res} }
\label{tab1}
\end{table}

\section{Conclusion remarks}
\label{sec-concl}
 Without assuming {\em relatively complete recourse}, we \linebreak
  prove the convergence of the SAA problem
(\ref{geq-7})-(\ref{geq-8}) of the two-stage SGE \eqref{geq-1}--\eqref{geq-2} in Theorem 2.4, and show the exponential rate of the convergence in Theorem 2.9. When the two-stage SGE \eqref{geq-1}--\eqref{geq-2} has  relatively complete recourse, Assumption~\ref{ass-3},  conditions (v)-(vi) in Theorem \ref{th-consist} and condition (iv) in Theorem \ref{t:expge} hold.

In section 3, we present  sufficient conditions for the existence, uniqueness, continuity and regularity of solutions of the two-stage  SVI-NCP \eqref{geq-SNCP-1}--\eqref{geq-SNCP-2} by using the perturbed linearization of functions $\Phi$ and $\Psi$ and then show the almost sure convergence and exponential convergence of its SAA problem \eqref{geq-vi-1-saa}-\eqref{geq-vi-2-saa}.  
Numerical examples in section 4 satisfy all conditions of Theorem \ref{t:expge} and we show the convergence of SAA method numerically.


\bibliographystyle{siamplain}
\bibliography{references}
\end{document}